\documentclass[12pt, reqno]{amsart}%
\usepackage[english]{babel}
\usepackage{amsthm}
\usepackage{amsmath}
\usepackage{amssymb}
\usepackage{dsfont}
\usepackage{graphicx}
\usepackage{esint}
\usepackage[shortlabels]{enumitem}
\usepackage{amscd,mdwlist,relsize}
\usepackage[margin=1in]{geometry}
\usepackage{color}
\usepackage{cite}
\usepackage{bbm, mathrsfs,pgf,tikz}
\usepackage{tcolorbox}
\usepackage[backgroundcolor=white, bordercolor=blue,
linecolor=blue]{todonotes}
\usepackage{url}
\usepackage{dsfont}
\usepackage{cases}
\usetikzlibrary{arrows}

\usepackage[pagebackref=false, pdfpagelabels, plainpages=false]{hyperref} 
\hypersetup{
colorlinks=true,
linkcolor=red,
filecolor=magenta,
urlcolor=cyan,
citecolor=blue,
}

\theoremstyle{plain}  
\newtheorem{theorem}{Theorem}[section]
\newtheorem{lemma}[theorem]{Lemma}

\theoremstyle{definition}  
\newtheorem{definition}[theorem]{Definition}
\newtheorem{example}[theorem]{Example}

\newtheorem{remark}[theorem]{Remark}   
\theoremstyle{remark}
\newtheorem*{remark*}{Remark}

\numberwithin{equation}{section}

\DeclareMathOperator{\loc}{loc}

\DeclareMathOperator{\pv}{\operatorname{p.\!v.}}

\DeclareMathOperator{\cE}{\mathcal{E}}

\DeclareMathOperator{\cJ}{\mathcal{J}}

\DeclareMathOperator{\R}{\mathbb{R}}

\renewcommand{\d}{\,\mathrm{d}}
\renewcommand{\div}{\operatorname{div}}

\newcommand{\vertiii}[1]{{\left\vert\kern-0.25ex\left\vert\kern-0.25ex\left\vert #1 \right\vert\kern-0.25ex\right\vert\kern-0.25ex\right\vert}}
\addto\extrasenglish{%

}
\makeatletter
\@namedef{subjclassname@2020}{\textup{2020} Mathematics Subject Classification}
\makeatother

\title{Optimal stability of Dirichlet problem for the regional fractional $p$-Laplacian}
\author{Guy Foghem}
\address{{\tiny Brandenburgische Technische Universit\"at Cottbus--Senftenberg, Fakult\"{a}t 1: MINT Fachgebiet Mathematik, Platz der Deutschen Einheit 1, 03046 Cottbus, Germany.} \href{https://orcid.org/0000-0002-8917-7309}{ORCID}}
\email{guy.foghem[at]b-tu.de}
\thanks{Financial support from the Deutsche Forschungsgemeinschaft (DFG) through the Walter Benjamin Programme (project FO~1699/1-1) is gratefully acknowledged.}

\keywords{Optimal stability, Fractional Sobolev spaces, Trace operator,  Regional fractional $p$-Laplace operator, Regional Dirichlet problem}
\subjclass[2020]{
26D15, 
35J20, 
35J92, 
46E35, 
46B70 , 
49J40, 
49J45  
}

\begin{document}

\begin{abstract}
We establish the optimal stability of Dirichlet boundary value problem for the regional (fractional) $p$-Laplacian $(-\Delta)^s_{p,\Omega}$ with $0<s\leq 1$, $\frac{1}{s}<p<\infty$ and $\Omega\subset \R^d$ bounded Lipschitz. More precisely, if $u_s \in W^{s,p}(\Omega)$ satisfies
$(-\Delta)^s_{p,\Omega} u_s = f_s$ in $\Omega$
and $u_s = g_s$ on $\partial \Omega$, then under appropriate condition on the date $f_s$ and $g_s$  we show that $\|u_s - u_1\|_{W^{s,p}(\Omega)} \to 0 \quad \text{as } s \to 1^-.$
 We also obtain an analogous optimal stability of the normalized Dirichlet eigenpairs $(\lambda_s,\varphi_s)$ associated with $(-\Delta)^s_{p,\Omega}$ .
\end{abstract}

\maketitle


\section{Introduction}
\label{sec:introduction}
 The analysis of nonlocal-to-local asymptotic as the fractional parameter $s \to 1^-$ has drawn considerable attention in recent years, serving as a bridge between nonlocal and classical variational problems. In this paper, we contribute to this line of research by analyzing \textit{the optimal convergence }of weak solutions $u_s\in W^{s,p}(\Omega)$ with $0<s\leq 1$ and $sp>1$ satisfying the Dirichlet problem
\begin{align}\label{eq:main-prob-regio-frac}
(-\Delta)_{p,\Omega}^s u_s = f_s \quad \text{ in }\,\, \Omega \quad \text{and} \quad \gamma_0^s(u_s) = g_s \quad \text{on }\,\, \partial\Omega.
\end{align}
Our second aim is to analyze the asymptotic of the Dirichlet eigenpairs $(\lambda_s,\varphi_s)$ associated with $(-\Delta)^s_{p, \Omega}$. Throughout this work, unless otherwise stated $\Omega \subset \R^d$, $d\geq 1$, denotes a bounded Lipschitz open set,  $1<p<\infty$, $sp>1$, $\gamma^s_0: W^{s,p}(\Omega)\to W^{s-\frac{1}{p},p}(\partial \Omega)$ is the trace operator, $f_s\in (W^{s,p}_0(\Omega))'$, $g_s\in W^{s-\frac{1}{p},p}(\partial \Omega)$, and  $(-\Delta)_{p,\Omega}^s$ for $0 < s \leq 1$ is the normalized regional fractional $p$-Laplacian defined for $x\in\Omega$

\begin{align*}
(-\Delta)^s_{p,\Omega}u(x)&:=C_{d,p,s}\pv \int_{\Omega}\frac{|u(x)-u(y)|^{p-2}(u(x)-u(y))}{|x-y|^{d+sp}}\d y\quad(0 < s <1),\\
(-\Delta)^1_{p,\Omega}u(x)&:= -\Delta_pu(x)= -\div(|\nabla u(x)|^{p-2} \nabla u(x))\qquad(s=1),
\end{align*}
where, as  in \cite{Fog25}, we consider the normalizing constant
\begin{align*}
C_{d,s,p}=\frac{s(1-2s)\Gamma(\frac{d+sp}{2})}{\pi^{\frac{d-1}{2}}\Gamma(\frac{1+s p}{2})\Gamma(p(1-s)) \cos(s\pi)}.
\end{align*}
Our interest in studying the asymptotic behavior of problem \eqref{eq:main-prob-regio-frac} is partly motivated by the  pointwise convergence  $(-\Delta)^s_{p,\Omega} u(x)\to (-\Delta)^1_{p,\Omega} u(x)$ as $s\to 1^-$, which holds for an appropriately regular function $u$; see \cite[Section 9]{Fog25}.
The constant $C_{d,s,p}$, verifies $C_{d,p,s}\asymp s(1-s)$ and exhibits the following asymptotic behaviors
\begin{align}\label{eq:asymp-normaliz-pconst}
\lim_{s\to 1^-} \frac{K_{d,p}C_{d,p,s}}{s(1-s)}=\lim_{s\to 0^+}\frac{\Gamma(p+1)C_{d,p,s}}{s(1-s)}=\lim_{s\to 0^+} \frac{2C_{d,2, \frac{sp}{2}}}{s(1-s)}=\frac{2p}{|\mathbb{S}^{d-1}|},
\end{align}
with the constant $K_{d,p}$ given by
\begin{align*}
K_{d,p}
&= \frac{1}{|\mathbb{S}^{d-1}|}\int_{\mathbb{S}^{d-1}} |w_d|^p \, d\sigma_{d-1}(w)
= \frac{\Gamma\big(\frac{d}{2}\big) \Gamma\big(\frac{p+1}{2}\big)}{\Gamma\big(\frac{d+p}{2}\big) \Gamma\big(\frac{1}{2}\big)}.
\end{align*}

The Sobolev-Slobodeckij  space $W^{s,p}(\Omega): =\{ u\in L^p(\Omega)\, :\, \|u\|_{W^{s,p}(\Omega)}<\infty\}$ with $s\in [0,1]$ and $1\leq p<\infty$ is equipped with the normalized norm
\begin{align*}
\|u\|_{W^{s,p}(\Omega)}&:=  \big(\|u\|^p_{L^p(\Omega)}+ [u]^p_{W^{s,p}(\Omega)} \big)^{1/p}\qquad s\in [0,1],
\end{align*}
As usual $W^{s,p}_0(\Omega)$ is the closure of $C_c^\infty(\Omega)$ in $W^{s,p}(\Omega)$ and is also given by $W^{s,p}_0(\Omega) = \{ u \in W^{s,p}(\Omega) : \gamma^s_0(u) = 0 \}$ since $\Omega$ is bounded Lipschitz. Here the fractional Gagliardo seminorm of order $s\in [0,1]$ is defined by
\begin{align*}
[u]^p_{W^{s,p}(\Omega)} :=
\begin{cases}
\frac{2|\mathbb{S}^{d-1}|}{p}\|u\|_{L^p(\Omega)}^p, & \text{if } s = 0, \\[2ex]
\displaystyle s(1-s) \iint_{\Omega\times\Omega}
\frac{|u(x) - u(y)|^p}{|x - y|^{d + sp}}\d y\d x, & \text{if } 0 < s < 1, \\
K_{d,p} \frac{|\mathbb{S}^{d-1}|}{p} \|\nabla u\|_{L^p(\Omega)}^p, & \text{if } s = 1.
\end{cases}
\end{align*}
The scaling factor $s(1-s)$ as highlighted in \cite{BBM01, MS02} and also in the earlier work \cite{AAS67} is crucial for our asymptotic analysis. In particular,  $[u]_{ W^{s,p}(\Omega)}\to[u]_{ W^{1,p}(\Omega)} $ as $s\to 1^-$.
\smallskip

Next, we need some notion of convergence for $f_s$ and $g_s$. The standard definition of weak convergence can not be directly applied to the functionals $f_s \in (W^{s,p}_0(\Omega))'$, as the underlying space $W^{s,p}_0(\Omega)$ varies dynamically with $s \to 1^-$. To overcome this difficulty, we introduce a notion of asymptotic weak convergence, formulated to mirror the classical properties of weak convergence in a fixed Banach space. Specifically, this framework ensures both the asymptotic boundedness of the sequence and a suitable variant of the weak–strong convergence lemma. To this end, we enforce the following natural assumptions on the data profiles $(f_s)_{s}$ and $(g_s)_{s}$.
\begin{enumerate}[label={(\Alph*)}, start=6]
\item 
\label{item:asymp-fs} \textbf{Asymptotic weak convergence:} The sequence of functionals $(f_s)_s$ with  $f_s \in(W^{s,p}_0(\Omega))'$ converges asymptotically weakly to a functional $f_1 \in (W^{1,p}_0(\Omega))'$ as $s \to 1^-$, meaning that the following two conditions are satisfied:
\medskip

\begin{enumerate}[$(F_1)$]\item \textbf{Asymptotic boundedness:} The sequence $(f_s)_s$ remains bounded as $s$ approaches $1$, that is, for some $s_0\in (0,1)$ we have
\begin{align*}\sup_{s\in (s_0,1)} \|f_s\|_{(W^{s,p}_0(\Omega))'} < \infty.
\end{align*}
\item \textbf{Weak--strong  property:} For any sequence $(v_s)_s$ with $v_s \in W^{s,p}_0(\Omega)$ and $v \in W^{1,p}_0(\Omega)$ such that $v_s \to v$ strongly in $L^p(\Omega)$ as $s \to 1^-$, we have
\begin{align*}\lim_{s \to 1^-} \langle f_s , v_s \rangle_s = \langle f_1, v\rangle_1,
\end{align*}
where $\langle \cdot, \cdot \rangle_s$ $0<s\leq 1$, denotes the duality pairing between $W^{s,p}_0(\Omega)$ and $(W^{s,p}_0(\Omega))'$.
In particular, choosing $v_s = v$ for a fixed $v \in W^{1,p}_0(\Omega)$ yields the standard weak convergence:
\begin{align*}\lim_{s \to 1^-} \langle f_s , v \rangle_s = \langle f_1 , v \rangle_1.
\end{align*}
\end{enumerate}
\item
\label{item:asymp-gs} \textbf{Asymptotic strong convergence of boundary data:} The sequence of boundary data $(g_s)_s$ with $g_s \in W^{s-\frac{1}{p},p}(\partial\Omega)$ converges asymptotically strongly to a limit function $g_1 \in W^{1-\frac{1}{p},p}(\partial\Omega)$ as $s \to 1^-$, that is,
\begin{align*}
\lim_{s\to1^-}	\|g_s-g_1\|_{W^{s-\frac{1}{p},p}(\partial\Omega) }=0.
\end{align*}
In particular, for each $\frac{1}{p}<\eta<1$, Lemma \ref{lem:frac-unif-bound-trace-bis} implies
\begin{align*}
\lim_{s\to1^-}\big(\|g_s-g_1\|_{L^p(\partial\Omega) }+ \|g_s-g_1\|_{W^{\eta-\frac{1}{p}, p}(\partial\Omega) }\big)=0.
\end{align*}
\end{enumerate}

Let us mention at one examples for each condition \ref{item:asymp-fs} and \ref{item:asymp-gs}. The condition \ref{item:asymp-fs} holds in particular if
$f_1, f_s \in L^{p'}(\Omega)$ where  $p'=p/(p-1)$ and the sequence
$(f_s)_s$ converges weakly to $f_1$ in $L^{p'}(\Omega)$ as $s\to 1^-$, that is,
\begin{align*}
\lim_{s\to1^-} \int_{\Omega} f_s(x)\, v(x)\,\d x
= \int_{\Omega} f_1(x)\, v(x)\,\d x
\qquad \text{for all } v \in L^{p}(\Omega).
\end{align*}
By the boundedness principle, the weak convergence of $(f_s)_s$ implies the boundedness
of the sequence $(\|f_s\|_{L^{p'}(\Omega)})_(s\in (s_0, 1))$ for some $s_0\in (0,1)$. Moreover,
\begin{align*}
\sup_{s\in(0,1)} \| f_s\|_{(W^{s,p}_0(\Omega))'}
\leq
\sup_{s\in(0,1)} \| f_s\|_{L^{p'}(\Omega)} < \infty.
\end{align*}
Now, if $(v_s)_s$ is a sequence such that $v_s \to v$ in $L^p(\Omega)$, then,
since $f_s \rightharpoonup f_1$ in $L^{p'}(\Omega)$, the standard weak--strong
convergence lemma yields
\begin{align*}
\lim_{s\to1^-} \int_{\Omega} f_s(x)\, v_s(x)\,\d x
= \int_{\Omega} f_1(x)\, v(x)\,\d x.
\end{align*}
Next, regarding condition \ref{item:asymp-gs}, we assume that $g_s, g_1 \in W^{1-\frac{1}{p},p}(\partial\Omega)$ are such that $\|g_s - g_1\|_{W^{1-\frac{1}{p},p}(\partial\Omega)} \to 0$ as $s \to 1^-$. Then  $(g_s)_s$ asymptotically strongly converges to $g_1$. Indeed, by Lemma \ref{lem:frac-unif-bound-trace-bis} there is  constant $C=C(d,p,\Omega)>0$  such that
\begin{align*}
\|g_s-g_1\|_{W^{s-\frac{1}{p},p}(\partial\Omega)}
\leq C\|g_s-g_1\|_{W^{1-\frac{1}{p},p}(\partial\Omega)}\xrightarrow{s\to 1^-}0.
\end{align*}
We are now in a position to state our first main result.
\begin{theorem}
\label{thm:opti-conv-dirch-frac-regio}
Assume $\Omega\subset \R^d$ is open bounded  Lipschitz and $1<p<\infty$.
Let  $ (f_s)_s$ with  $f_s\in  (W^{s,p}_{0}(\Omega))'$ converge asymptotically weakly to $f_1\in (W^{1,p}_0(\Omega))'$ and $ (g_s)_s$ with $g_s\in  W^{s, p}(\Omega)$ converge asymptotically strongly to $g_1\in W^{1,p}(\Omega)$.
Let $u_s\in W^{s,p}(\Omega)$ with  $\frac{1}{p}<s\leq 1$, be the  weak solution of  the Dirichlet problem \eqref{eq:main-prob-regio-frac}.
\noindent Then the sequence $(u_s)_s$ strongly converges to $u_1$ in the optimal sense, that is
\begin{align*}
\lim_{s\to1^-}\|u_s -u_1\|_{W^{s,p}(\Omega)} =0.
\end{align*}
\end{theorem}
A notable difficulty in establishing the optimal convergence of weak solutions $(u_s)_s$ stems from the fact that the asymptotic behavior of the weak solutions to \eqref{eq:main-prob-regio-frac} near the boundary $\partial\Omega$ is particularly challenging to analyze. An essential tool to address this obstacle is the robust interpolation inequality established in \cite{Fog26rob}, which yields the compactness of $(u_s)_s$ in any $W^{\eta, p}(\Omega)$, especially for $\frac{1}{p}<\eta<1$; see for instance Theorem \ref{thm:asymp-compact-frac}. The compactness in the higher regularity regime, $\frac{1}{p}<\eta<1$, directly implies the compactness of the trace sequence $(\gamma^s_0(u_s))_s$ in $W^{\eta-\frac{1}{p}, p}(\partial \Omega)$. In the spirit of Theorem \ref{thm:opti-conv-dirch-frac-regio} we also  establish the stability of the Dirichlet eigenpairs associated with $(-\Delta)^s_{p,\Omega}$.
\begin{theorem}
\label{thm:optimal-conv-eigpair}
Let $\Omega \subset \R^d$ be an open bounded Lipschitz domain and $1 < p < \infty$.
Let  $(\lambda_s, \varphi_s) \in \R \times W^{s,p}_0(\Omega)$, $\frac{1}{p}<s<1,$ be a normalized Dirichlet eigenpairs of $(-\Delta)^s_{p,\Omega}$ in the weak sense, that is, $\|\varphi_s\|_{L^p(\Omega)} = 1$ and it satisfies
\begin{align*}
(-\Delta)^s_{p,\Omega} \varphi_s= \lambda_s|\varphi_s|^{p-2}\varphi_s\quad\text{in}\,\, \Omega\quad \text{and}\quad \gamma^s_0(\varphi_s)=0\quad\text{on}\,\, \partial\Omega.
\end{align*}
Then there exist a subsequence  $s_n \to 1^-$  and a normalized Dirichlet eigenpair  $(\lambda_1, \varphi_1)\in\R\times  W^{1,p}_0(\Omega)$ of
$-\Delta_p$ such that  as $n \to \infty$  we have
\begin{align*}
\lambda_{s_n} \to \lambda_1 \qquad \text{in } \R \quad \text{and}\quad \|\varphi_{s_n} - \varphi_1\|_{W^{s_n,p}(\Omega)} \to 0.
\end{align*}
\end{theorem}
\smallskip

The proof of the optimal convergences in Theorem \ref{thm:opti-conv-dirch-frac-regio} and Theorem \ref{thm:optimal-conv-eigpair}  require several key ingredients. The first of these is the study of the robustness of the trace operator $\gamma^s_0$ as $s\to 1^-$, and consequently, the asymptotics of the fractional space $W^{s-\frac{1}{p},p}(\partial\Omega)$. In fact, a further objective of this work is to study the robustness of the nonlocal trace spaces
$W^{s-\frac{1}{p},p}(\partial \Omega)$ in the limit $s \to 1^-$.  Specifically, in the spirit of  \cite{BBM01,Pon04} we  establish in Theorem \ref{thm:asymp-compact-frac-trace} the asymptotic compactness $(W^{s-\frac{1}{p},p}(\partial \Omega))_s$. Namely for $(g_s)_s$ such that $g_s\in W^{s-\frac{1}{p},p}(\partial \Omega)$ if $\sup_{s\in (\frac{1}{p},1)}\|g_s\|_{W^{s-\frac{1}{p},p}(\partial \Omega)}<\infty$ then there is $g\in W^{1-\frac{1}{p},p}(\partial \Omega)$ and subsequence $s_n\to 1$ such that $\|g_{s_n}-g\|_{W^{\tau-\frac{1}{p},p}(\partial \Omega)}\xrightarrow{s_n\to 1^- }0$ for all $\frac{1}{p}<\tau<1$.
In addition, we obtain in Theorem \ref{thm:charac-frac-trace} another characterization of $W^{1-\frac{1}{p},p}(\partial\Omega)$. Namely that, a function $g \in L^p(\partial\Omega)$ is belongs to $W^{1-\frac{1}{p},p}(\partial \Omega)$ provided that
\begin{align*}
\liminf_{s \to 1^-} \|g\|_{W^{s-\frac{1}{p},p}(\partial \Omega)} < \infty.
\end{align*}
Another essential ingredient is the robust Friedrichs-Poincar\'e inequality. We take this opportunity to establish in Section \ref{sec:gene-fract-poin-ineq} several variants of these robust Poincar\'e type inequalities, some of which may be of independent interest. In particular, we show in Theorem \ref{thm:rob-friedrichs} that for a given subset $\Gamma_0 \subset \partial\Omega$ with positive Hausdorff measure, if $\Omega$ is a bounded, connected Lipschitz domain, then there exist two constant $s_0 = s_0(d, p, \Omega, \Gamma_0) \in (\frac{1}{p}, 1)$ and $C = C(d, p, \Omega, \Gamma_0) > 0$ such that,
\begin{align*}
\|u\|_{L^p(\Omega)} \leq C [u]_{W^{s,p}(\Omega)} + C \Big| \int_{\Gamma_0} \gamma^s_0(u )\d\sigma \Big|\quad \text{for all $s \in (s_0, 1)$, $u \in W^{s,p}(\Omega)$}.
\end{align*}
Thus, by  Theorem \ref{thm:rob-friedrichs-zero},
if we set $W^{s,p}_{\Gamma_0}(\Omega) :=
\{ u \in W^{s,p}(\Omega) : \gamma^s_0(u)|_{\Gamma_0} = 0 \}$,  then
\begin{align*}
\|u\|_{L^p(\Omega)} \leq C [u]_{W^{s,p}(\Omega)}\quad \text{for all $s \in (s_0, 1)$, $u \in  W^{s,p}_{\Gamma_0}(\Omega)$}.
\end{align*}

It is worth emphasizing that, see Theorem \ref{thm:robust-poinca-fried-regio}, the connectedness assumption on $\Omega$ can be omitted in the special case where $\Gamma_0 = \partial\Omega$.
To obtain, these various Poincare inequalities we benefit from the asymptotic compactness and the robust interpolation inequality established in \cite{Fog26rob}.
Finally, we also establish the convergence of the forms $\cE^{s,p}_\Omega(\cdot,\cdot)$ and their associated energy functionals in the sense of $\Gamma$-convergence.
\medskip

 Let us  now provide a brief account of the literature.
To the best of our knowledge, the convergence of weak solutions to the Dirichlet problem associated with the regional fractional $p$-Laplacian $(-\Delta)^s_{p, \Omega}$ is new. The optimal convergence obtained  in Theorem \ref{thm:opti-conv-dirch-frac-regio} improves upon our previous finding \cite[Theorem 1.4]{Fog26rob}, where we only established the weaker convergence $\|u_s-u_1\|_{W^{\eta, p}(\Omega)}\to 0$ as $s\to 1^-$, for all $0\leq \eta<1$ under relaxed assumptions on $f_s$ and $ g_s$.
The convergence of weak solutions to the Dirichlet problem associated with the regional fractional $p$-Laplacian and even for regional operators of $p$-L\'evy type can be found in \cite{Fog25,Fog26rob}. An abundant amount of literature addresses the convergence of weak solutions for the full fractional $p$-Laplacian $(-\Delta)^s_p \equiv (-\Delta)^s_{p, \R^d}$.
For instance, in the setting $p=2$, the convergence rate for the Dirichlet problem is discussed in \cite{BuFe25}, the differentiability of the mapping $s\mapsto u_s$ is investigated in \cite{JSW25,JSW20}. We also refer the reader to \cite{Voi17,guy-thesis, FoKa24,Gru25} wherein the asymptotic for general nonlocal elliptic operators is treated.  For the general case $p\neq 2$, we refer the reader to the series of works \cite{BPS16,FeSa20,BS22,BO20,BO21,SaVe22} which study the convergence from nonlocal to local of weak solutions to homogeneous Dirichlet problems associated with the fractional $p$-Laplacian $(-\Delta)^s_p$, see also \cite{Voi17, guy-thesis,Fog25, Fog26} for the case of inhomogeneous Dirichlet problems associated with operators of L\'evy type.  The convergence of nonlocal Neumann problems associated with regional nonlocal  operators  of $p$-L\'evy type can be found in \cite{AMRT08, AMRT10,DLS15,Fog25,Fog26}.  Convergence  of weak solutions to nonlocal  problems with complement Neumann condition to the local ones is established in \cite{guy-thesis,FoKa24,GrHe24,Fog25,Fog26}.

The rest of this paper is organized as follows. In Section \ref{sec:miscellaneous}, we collect some necessary preliminaries results. Section \ref{sec:robust-trace} is devoted to the proof of the robust trace operator  and the  asymptotics of the trace space. In Section \ref{sec:gene-fract-poin-ineq} we establish various robust Poincar\'e inequalities. Finally in Section \ref{sec:opti-conv-regio-frac}, we establish our main optimal convergence results.

\section{Miscellaneous}\label{sec:miscellaneous}
In this section, we collect several useful auxiliary results.
\subsection{Pointwise convergence}
We begin by discussing the motivation behind the choice of the normalization constant $C_{d,s,p}$ introduced in \cite{Fog25}, and compare it with alternative constants considered in \cite{dTGCV21,DJS25}. In view of \eqref{eq:asymp-normaliz-pconst}, this specific choice of $C_{d,p,s}$ guarantees consistent asymptotic behavior near both $s=1$ and $s=0$. First, for $u \in C^2(B_1(x)) \cap L^\infty(\R^d)$ satisfying $|\nabla u(x)| \neq 0$ with $1 < p < 2$, we have (see \cite[Section 9]{Fog25}):
\begin{align*}
\lim_{s \to 1^-} C_{d,p,s} \pv \int_{\R^d} \frac{|u(x)-u(y)|^{p-2}(u(x)-u(y))}{|x-y|^{d+sp}} \d y = \Delta_p u(x).
\end{align*}
Likewise, at the nonlocal limit $s \to 0^+$, for $u \in C^\varepsilon(B_1(0)) \cap L^{p'}(\R^d)$ with $0 < \varepsilon \leq 1$ such that $u(x) \to 0$ as $|x| \to \infty$, it follows from \cite[Theorem 4.9]{Fog26rob} that:
\begin{align*}
\lim_{s \to 0^+} C_{d,2,\frac{sp}{2}} \pv \int_{\R^d} \frac{|u(x)-u(y)|^{p-2}(u(x)-u(y))}{|x-y|^{d+sp}} \d y = |u(x)|^{p-2} u(x).
\end{align*}
Moreover, this normalization yields, see \cite[Corollary 4.4]{Fog26rob}, the  Brezis--Bourgain--Mironescu \cite{BBM01} and Maz'ya--Shaposhnikova \cite{MS02} formulas:

\begin{align*}
\lim_{s \to 1^-} \frac{C_{d,p,s}}{2} \iint_{\R^d\times \R^d} \frac{|u(y)-u(x)|^p}{|x-y|^{d+sp}} \d y \, \d x &= \int_{\R^d}|\nabla u(x)|^p \d x \quad \text{for } u \in W^{1,p}(\R^d), \\[1ex]
\lim_{s \to 0^+} \frac{C_{d,2,\frac{sp}{2}}}{2} \iint_{\R^d\times \R^d} \frac{|u(y)-u(x)|^p}{|x-y|^{d+sp}} \d y \, \d x &= \int_{\R^d}|u(x)|^p \d x \quad \text{for } u \in \bigcup_{0 < s < 1} W^{s,p}(\R^d).
\end{align*}
According to \cite{Fog26, Fog23},  if $\Omega$ is a $W^{1,p}$-extension domain then for all $u,v \in W^{1,p}(\Omega)$,
\begin{align}\label{eq:xasymp-form}
\begin{split}
\lim_{s\to 1^-} \cE^{s,p}_{\Omega}(u,v) = \cE^{1,p}_{\Omega}(u,v),\qquad 1<p<\infty,\\
\lim_{s\to 1^-} \cE^{s,p}_{\Omega}(u,u) = \cE^{1,p}_{\Omega}(u,u),\qquad 1\leq p<\infty,
\end{split}
\end{align}
where $\cE^{s,p}_\Omega(\cdot,\cdot)$ is the form associated with $(-\Delta)_{p,\Omega}^s$ are defined by
\begin{align*}
\cE^{s,p}_{\Omega}(u,v) &= \frac{C_{d,p,s}}{2} \iint_{\Omega \times \Omega} \frac{|u(y)-u(x)|^{p-2}(u(y)-u(x))}{|x-y|^{d+sp}} (v(y)-v(x)) \d y \, \d x, \\
\cE^{1,p}_{\Omega}(u,v) &= \int_\Omega |\nabla u(x)|^{p-2}\nabla u(x) \cdot \nabla v(x) \d x.
\end{align*}
We emphasize that according to \cite[Theorem 4.6]{Fog26rob}, for any open set $\Omega \subset \R^d$ and any function $u \in \bigcup_{0<s<1} \left( W^{s,p}(\Omega) \cap L^p(\Omega, \delta_x^{-sp}) \right)$ with $\delta_x = \operatorname{dist}(x,\partial\Omega)$ and $1 \leq p < \infty$, the global profile naturally encodes the geometry of the domain:
\begin{align*}
\lim_{s\to0^+} \frac{C_{d,2,\frac{sp}{2}}}{2} \iint_{\Omega\times\Omega} \frac{|u(x)-u(y)|^p}{|x-y|^{d+sp}} \d y \, \d x = \int_\Omega \left(1-\frac{|S_\Omega(x)|}{|\mathbb{S}^{d-1}|}\right) |u(x)|^p \d x,
\end{align*}
where $S_\Omega(x) = \big\{w \in \mathbb{S}^{d-1}: \lim\limits_{r\to\infty} \mathbf{1}_{\Omega^c}(x+rw)=1 \big\}$ is the escape sector at $x \in \Omega$.

\subsection{Robust interpolation inequalities}
An open set $\Omega \subset \R^d$ is a robust $W^{s,p}$-extension domain  if there exists a  linear operator $E: L^p(\Omega) \to L^p(\R^d)$ and a positive constant $C = C(d, p, \Omega) > 0$ such that: $Eu\vert_{\Omega} = u \quad \text{a.e. in } \Omega$, for every $u \in L^p(\Omega)$, and  for every $s \in [0, 1]$ we have
\begin{align*}
\|Eu\|_{W^{s,p}(\R^d)} \leq C \|u\|_{W^{s,p}(\Omega)}\qquad\text{ for all $u \in W^{s,p}(\Omega)$}.
\end{align*}
We need the following robust interpolation inequality, borrowed from \cite{Fog26rob}.
\begin{theorem}\label{thm:robust-interpo-dom}
 Assume that $\Omega \subset \R^d$ is a robust $W^{s,p}$-extension domain; this holds true, in particular, if $\Omega = \R^d$, $\Omega = \R^d_+$, or if $\partial\Omega$ is compact and Lipschitz. Let  $1\leq  p\leq \infty$. There is a positive constant  $C = C(d, p, \Omega)>0$ such that: if $0 \leq \sigma < \eta < s \leq 1$ with $\eta = \theta s + (1-\theta)\sigma,$ $\theta \in (0,1)$ and $u \in W^{s,p}(\Omega)$ then we have
\begin{align*}
\|u\|_{W^{\eta,p}(\Omega)} \leq C \|u\|_{W^{s,p}(\Omega)}^{\theta} \|u\|_{W^{\sigma,p}(\Omega)}^{1-\theta}.
\end{align*}
\end{theorem}
Theorem \ref{thm:robust-interpo-dom} is a direct consequence of the following one.

\begin{theorem}
\label{thm:robust-interpo-esti}
Let $u\in L^p(\mathbb{R}^d)$ with $1\leq p\leq \infty$ and let $0\leq \sigma<\eta<s\leq 1$ say $\eta=\theta s+(1-\theta)\sigma,$ with $\theta=\frac{\eta-\sigma}{s-\sigma}\in(0,1).$ Then the following the estimate holds
\begin{align*}
[u]_{W^{\eta,p}(\R^d)}
&\leq
2^{1+\frac{2}{p}}
\,[u]_{W^{s,p}(\R^d)}^{\theta}
\cdot [u]_{W^{\sigma,p}(\R^d)}^{1-\theta}.
\end{align*}
\end{theorem}
We take this opportunity to  provide an answer to the  open question from \cite[Section 5.4]{Fog26rob} in the case $p=2$ by establishing an alternative proof to \cite[Lemma 5.2]{Fog26rob}.
\begin{lemma}\label{lem:tail-short-cos}
Let  $\xi\in\R$ and $0\leq \sigma\leq \eta<1$. The following estimates hold true
\begin{align*}
\eta \int_1^\infty \frac{(1-\cos (\xi t))}{t^{1+2\eta}}\d t
&\leq 2\sigma  \int_1^\infty \frac{(1-\cos (\xi t))}{t^{1+2\sigma}}\d t.
\end{align*}
\end{lemma}

\begin{proof}
Note that the Gamma family $(\Gamma(2\eta, \cdot))_\eta$ with  $\Gamma(2\eta, x) =\frac{x^{2\eta-1}}{\Gamma(2\eta)} e^{-x}\mathds{1}_{(0,\infty)}(x)$ of probability densities function of  $(2\eta,1)$-Gamma distributions $X_\eta$
possesses the monotone likelihood ratio property in $x$, i.e.,  for $\sigma<\eta$, the mapping
\begin{align*}
x\mapsto \frac{\Gamma(2\eta, x)}{\Gamma(2\sigma, x)}= \frac{\Gamma(2\sigma)}{\Gamma(2\eta)}x^{2(\eta-\sigma)}
\quad\text{is increasing on $(0,\infty)$}.
\end{align*}
Let us observe that
\begin{align*}
\eta \int_1^\infty \frac{(1-\cos (\xi t))}{t^{1+2\eta}}\d t
&= \eta \int_{1}^{\infty}  (1-\cos(\xi t))
\left( \frac{1}{\Gamma(1+2\eta)} \int_{0}^{\infty} x^{2\eta} e^{-xt} \, \d x \right)  \d t \\
&=  \frac{1}{2} \int_{0}^{\infty} \frac{x^{2\eta-1} }{\Gamma(2\eta)}  \left( \int_{1}^{\infty} (1-\cos(\xi t)) xe^{-xt} \, \d t \right)  \d x \\
&=  \frac{1}{2} \int_{0}^{\infty} K(x,\xi) \Gamma(2\eta, x)  \d x
\end{align*}
where we consider the kernel
\begin{align*}
K(x,\xi)=e^x \operatorname{Re}\left( \int_{1}^{\infty} (1-e^{i\xi t})  xe^{-xt} \, \d t \right) =
1- \frac{x^2\cos \xi -x\xi  \sin\xi}{x^2+\xi^2 }.
\end{align*}
The sought inequality becomes
\begin{align*}
\int_{0}^{\infty} K(x,\xi) \Gamma(2\eta, x)  \d x \leq2 \int_{0}^{\infty} K(x,\xi) \Gamma(2\sigma, x)  \d x.
\end{align*}
Therefore, according to 
\cite[Theorem 2.20]{Fog26rob} the desired estimate is valid once we show that, each $K(\cdot, \xi)$ is $2$-almost decreasing. That is, it remains to prove that
\begin{align*}
K(y,\xi)\leq 2 K(x,\xi)\qquad\text{ for all }\quad x\leq y.
\end{align*}
\textbf{Case 1: $\cos\xi\ge 0$ and $\sin\xi\ge 0$.}
Note that Cauchy-Schwartz inequality yields $-\xi\sin\xi + x\cos\xi
\leq \sqrt{x^2+\xi^2}$. It follows that
\begin{align*}
\frac{x\xi\sin\xi}{x^2+\xi^2}
\leq
\frac{x}{\sqrt{x^2+\xi^2}}- \frac{x^2\cos\xi}{x^2+\xi^2}
\leq 1-\frac{x^2\cos\xi}{x^2+\xi^2}:= g(x,\xi).
\end{align*}
Since $\sin\xi\geq 0$ if follows that
\begin{align*}
g(x, \xi) \leq K(x,\xi)= g(x,\xi) +
\frac{x\xi\sin\xi}{x^2+\xi^2}\leq 2g(x, \xi)
\end{align*}

Therefore $K(\cdot,\xi)$ is $2$-almost decreasing, i.e.,  we have
\begin{align*}
K(y,\xi)\leq 2K(x,\xi)\qquad x\leq y.
\end{align*}
\textbf{Case 2: $\cos\xi\leq 0$ and $\sin\xi\geq 0$.}
Since $\cos\xi<0$ we have
\begin{align*}
1\leq 1- \frac{x^2\cos (\xi) -x\xi  \sin(\xi)}{x^2+\xi^2 }= K(x,\xi)\leq 1+ \frac{x}{\sqrt{x^2+\xi^2} }\leq 2.
\end{align*}
Therefore we have
\begin{align*}
K(x,\xi)\leq 2 K(y,\xi) \qquad x,y\in( 0, \infty).
\end{align*}

\textbf{Case 3: $\sin\xi\leq 0$.}
Since $\sin\xi\leq 0$, solving for $x$ the equation
\begin{align*}
\partial_xK(x,\xi)
=
-\frac{\xi}{(x^2+\xi^2)^2}
\Big[
x^2\sin\xi +2x\xi\cos\xi -\xi^2\sin\xi
\Big]= 0
\end{align*}
leads to  the unique global minimum point $x_0
=- \frac{ \xi(1+\cos \xi)}{\sin \xi}.$  Moreover, we have
\begin{align*}
K(0,\xi) = 1,\quad
K(x_0,\xi) = \frac{1-\cos\xi}{2},\quad \lim_{x\to\infty} K(x,\xi) = 1-\cos\xi.
\end{align*}
This implies that
for all $x,y\in (0,\infty)$ we have $K(y,\xi) \leq 2K(x,\xi) $ since
\begin{align*}
\frac{1-\cos\xi}{2}\leq K(x,\xi)\leq  1-\cos\xi\leq 2 K(y,\xi).
\end{align*}
Hence $K(\cdot,\xi)$ is $2$-almost decreasing, i.e.,
\begin{align*}
K(y,\xi)\leq 2K(x,\xi)\qquad\text{for all $0\leq x\leq y$}.
\end{align*}
\textbf{Case 4: $\cos\xi\leq 0$ and $\sin\xi\leq 0$.}
Using $0\leq \frac{x\xi}{x^2+\xi^2}\leq \frac{1}{2},$ we find that
\begin{align*}
\left|\frac{x\xi\sin\xi}{x^2+\xi^2}\right|
\leq \frac{1}{2} |\sin\xi|
\leq \frac{1}{2},
\end{align*}
It follows that
\begin{align*}
g(x,\xi)-\frac{1}{2}\leq K(x,\xi) = g(x,\xi) + \frac{x\xi\sin\xi}{x^2+\xi^2}\leq g(x,\xi).
\end{align*}
Since $\cos\xi\leq 0$,  $x\mapsto g(x,\xi)$ is increasing and we have
\begin{align*}
1=g(0,\xi)\leq g(x,\xi)
= 1-\frac{x^2\cos\xi}{x^2+\xi^2}
= 1-\cos\xi + \frac{\xi^2\cos\xi}{x^2+\xi^2}.
\end{align*}
Given that $g(x,\xi)\ge 1$, we obtain $g(x,\xi)-\frac{1}{2} \geq \frac{1}{2} g(x,\xi)$ and hence
\begin{align*}
\frac{1}{2} g(x,\xi)\leq K(x,\xi)\leq g(x,\xi).
\end{align*}
 This clearly implies $K(y,\xi)\leq 2K(x,\xi)$ for all $x,y\geq 0$.
\end{proof}
As direct consequence we obtain the following theorem.
\begin{theorem}
\label{thm:r-split-monot-semin-l2}
Let $u\in L^2(\R^d)$ and $r>0$. The following estimates  hold
\begin{align*}
T_2(\eta, u,r)&\leq 2^{\frac{1}{2}} T_2(\sigma , u,r), && 0\leq \sigma< \eta \leq 1
\end{align*}
with
\begin{align*}
T^2_2 (\eta, u,r)
&=\eta\, r^{2\eta}\int_r^\infty \int_{ \mathbb{S}^{d-1}} \int_{\R^d} | u(x+\omega t)-u(x)|^2\d x\d \sigma_{d-1}(\omega)  \frac{\d t}{t^{1+2\eta}}.
\end{align*}
\end{theorem}
\begin{proof}
The desired estimate  follows from  Lemma \ref{lem:tail-short-cos} since Plancherel's identity for the Fourier  transform implies
\begin{align*}
T^2_2 (\eta, u,r)&=\eta\, r^{2\eta}\int_r^\infty \frac{\d t}{t^{1+2\eta}} \int_{ \mathbb{S}^{d-1}} \int_{\R^d} | u(x+\omega t)-u(x)|^2\d x\d \sigma_{d-1}(\omega)\\
&=2 \int_{\mathbb{S}^{d-1}}  \int_{\R^d}  |\widehat{u}(\xi)|^2 \, \, \eta  r^{2\eta}\int_r^\infty \frac{(1-\cos (|\omega\cdot \xi| t))}{t^{1+2\eta}}\d t \d \xi \d \sigma_{d-1}(\omega).
\end{align*}
\end{proof}

\subsection{Fractional trace theorem} \label{sec:trace-theorem}
It is well-known that, see for instance \cite{DK21},    \cite[Chapter 9]{Leo23}, if $sp\leq 1$ then $W^{s,p}_0(\Omega)=W^{s,p}(\Omega)$. Therefore, the  trace operator on  $W^{s,p}(\Omega)$  is only worth studying when $sp>1$. For the reader’s convenience, we recall the trace theorem for $W^{s,p}(\Omega)$ with $0 < s \leq 1$. The fractional case is treated, for instance, in \cite{DMT22} and \cite[Chapter~9]{Leo23}, while for the classical case $s = 1$ we refer the reader to  the version in  \cite[Chapter~III]{BF13}.
\begin{theorem}[{Trace theorem for $W^{s,p}$}]
\label{thm:trace-frac-thm}
Let $\Omega \subset \R^d$ be a bounded Lipschitz domain. Assume $1 \leq p < \infty$ and $0 < s \leq 1$ such that $sp > 1$ or $p=s=1$. Then, there exists a bounded linear trace operator $\gamma^s_0: W^{s,p}(\Omega) \to L^p(\partial\Omega)$ such that $\gamma^s_0(u) = u|_{\partial\Omega}$ for all $u \in C^1(\overline{\Omega}) \cap W^{s,p}(\Omega)$. Namely there is a constant $C_s=C(d,p,s,\Omega)>0$ such that
\begin{align*}
\|\gamma^s_0(u)\|_{ L^p(\partial \Omega)}\leq
C_s \|u\|_{W^{s,p}(\Omega)}.
\end{align*}
Furthermore, the following assertions are true.
\begin{enumerate}[$(a)$]
\item  $ \ker(\gamma^s_0 ) = W^{s,p}_0(\Omega)$, and  $\gamma^s_0 (W^{s,p}(\Omega))\overset{def}{=} W^{s-\frac{1}{p},p}(\partial \Omega)$.
\item The mappings $W^{s,p}(\Omega) \xrightarrow{\gamma^s_0}W^{s-\frac{1}{p},p}(\partial\Omega)
\xrightarrow{\operatorname{Id}} L^p(\partial \Omega)$
are continuous.
\item The operator $\gamma^s_0$ has a right  lifting, i.e., there is $R^s_0:W^{1-\frac{1}{p},p}(\partial\Omega)\to W^{1,p}(\Omega)$ bounded and continuous mapping such that $ \gamma^s_0\circ R^s_0= Id$.
\end{enumerate}
In addition in the case $s=1$ we have
\begin{align*}
\|\gamma^1_0(u)\|_{ L^p(\partial \Omega)}\leq C\|u\|^{1-\frac{1}{p}}_{L^p(\Omega)}
\|u\|^{\frac{1}{p}}_{W^{1,p}(\Omega)} \quad\text{for all $u\in W^{1,p}(\Omega)$}.
\end{align*}
\end{theorem}

\noindent Let us make some comments about the trace space $W^{s-\frac{1}{p}, p}(\partial \Omega)=\gamma^s_0( W^{s,p}(\Omega))$.
For practical reasons, we equip
$W^{s-\frac{1}{p} ,p}(\partial \Omega)$ with the norm
\begin{align*}
\|g\|_{W^{s-\frac{1}{p} ,p}(\partial \Omega)}
=
\inf\big\{ \|u\|^*_{W^{s,p}(\Omega)} \,\,:\,\, \gamma^s_0(u)=  g,\quad u\in W^{s,p}(\Omega) \big\}.
\end{align*}
where we consider the normalized norm
\begin{align*}
\|u\|^*_{W^{s,p}(\Omega)} &= \Big( \int_\Omega |u(x)|^p\d x+ \frac{C_{d,p,s}}{2} \iint_{ \Omega \times \Omega} \frac{|u(x)-u(y)|^p}{|x-y|^{d+sp}}\d y\d x\Big)^{\frac{1}{p}}.
\end{align*}
Note that, since $s\mapsto\frac{C_{d,s,p}}{s(1-s)}$ is bounded, there is a constant $c=c(d,p)>0$ such that
\begin{align}\label{eq:equiv-norm-start}
c^{-1}\|u\|_{W^{s,p}(\Omega)}\leq \|u\|^*_{W^{s,p}(\Omega)}\leq c \|u\|_{W^{s,p}(\Omega)}.
\end{align}
Furthermore, $\|u\|^*_{W^{1, p}(\Omega)}= \|u\|_{W^{1, p}(\Omega)}$ since by \eqref{eq:xasymp-form} we find that
$\|u\|^*_{W^{s,p}(\Omega)}\xrightarrow{s\to 1^-} \|u\|_{W^{1, p}(\Omega)}.$
 In particular, we get
\begin{align*}
\|g\|_{W^{1-\frac{1}{p} ,p}(\partial \Omega)}= \inf\big\{ \|u\|_{W^{1,p}(\Omega)} \,:\, \gamma^1_0(u)= g
\,\,\text{with} \quad u\in W^{1,p}(\Omega)\big\}.
\end{align*}
It is important to emphasize that
a lifting operator $ R^s_0 $ can be chosen to be linear (though not necessarily) only when
$1 < p < \infty $, whereas $R^1_0$ can never be linear in the case $ p=s = 1$. This is known as the Peetre's theorem \cite{Pee79}; a more recent and accessible proof can be found \cite{PeWo02}. Let us recall that when $\Omega$ is bounded Lipschitz, $s=1$ and  $1<p<\infty$, it is possible to obtain the explicit equivalence of the following trace norm, see for instance \cite{MiRu15},
\begin{align*}
\|u\|_{W^{1-\frac{1}{p} ,p}(\partial\Omega)}\asymp\Big( \int_{\partial \Omega} |u(x)|^p\d \sigma(x)+  \iint_{\partial \Omega\times \partial \Omega} \frac{|u(x)-u(y)|^p}{|x-y|^{d+ p-2}}\d \sigma(y) \d \sigma(x)\Big)^{\frac{1}{p}}.
\end{align*}
Interestingly, according to \cite{Gag57}, it turns out that in the case $p=1$,   the trace space $W^{0,1}(\partial\Omega)$ of $W^{1,1}(\Omega)$ coincides with $L^1(\partial\Omega)$, i.e. $\gamma_0 (W^{1,1}(\Omega))= L^1(\partial\Omega)$; we strongly recommend see \cite{Mir15} for elegant proof of the latter statement.

\begin{remark}\label{rem:trace-mono}
Let us point out two important remarks.
\begin{enumerate}[$(i)$]
\item
Since $W^{s,p}(\Omega)\subset W^{\eta,p}(\Omega)$ for $\frac{1}{p}<\eta<s\leq 1$, the spaces $W^{s,p}(\Omega)$ and $W^{\eta,p}(\Omega)$ share the same dense subset $C^1(\overline{\Omega}) \cap W^{s,p}(\Omega)$. Consequently, it follows that
$\gamma^s_0 = \gamma_0^\eta \big|_{W^{s,p}(\Omega)}$ for $ \frac{1}{p}<\eta<s\leq 1.$
\item In practice, when $\Omega$ is bounded, the estimate (see \cite[Chapter~III]{BF13})
\begin{align*}
\|\gamma^1_0(u)\|_{ L^p(\partial \Omega)}\leq C\|u\|^{1-\frac{1}{p}}_{L^p(\Omega)}
\|u\|^{\frac{1}{p}}_{W^{1,p}(\Omega)},
\end{align*}
allows to show that the trace operator $\gamma^1_0:W^{1,p}(\Omega)\to L^p(\partial \Omega)$ is compact. Interestingly, for the case $0<s<1$ with $sp>1$, if we consider $0<\eta<1$ such that $sp>\eta p>1$ say $\eta= (1-\tau) s+ \frac{\tau}{p}$ for any fixed $0<\tau<1$. Since  $\gamma^s_0 = \gamma_0^\eta \big|_{W^{s,p}(\Omega)}$, applying   Theorem \ref{thm:robust-interpo-dom} with $\theta= \frac{\eta}{s}$, we find that
\begin{align*}
\|\gamma^s_0(u)\|_{ L^p(\partial \Omega)}
\leq C_\eta
\|u\|_{W^{\eta,p}(\Omega)}
\leq C_\eta \|u\|^{1-\frac{\eta}{s}}_{L^{p}(\Omega)}\|u\|^{\frac{\eta}{s}}_{W^{s,p}(\Omega)}.
\end{align*}
That is, for any fixed $0<\tau<1$ we have $\theta= \frac{\eta}{s}= (1-\tau) + \frac{\tau}{sp}\in (0,1)$ and
\begin{align*}
\|\gamma^s_0(u)\|_{ L^p(\partial \Omega)}\leq
C_s
\|u\|^{1-\theta}_{L^{p}(\Omega)}\|u\|^{\theta}_{W^{s,p}(\Omega)},
\end{align*}
where $C_s= C(d,p,s,\Omega)>0$. However, we believe a better exponent factor $\theta$ can expected in the latter inequality.
\end{enumerate}
\end{remark}

\subsection{Asymptotic compactness of Lions--Peetre type}

\begin{theorem}\label{thm:asymp-converg-interpo}
Let $\Omega \subset \R^d$ be any open set and $1\leq p<\infty$. Assume  $(u_s)_{s\in (0,1)}\subset L^p(\Omega)$ is asymptotically bounded, that is, for some $s_0\in (0,1)$ we have
\begin{align*}
\sup_{s\in(s_0,1)}
\Big(\int_\Omega |u_s(x)|^p\d x+ \frac{C_{d,p,s}}{2}
\iint_{\Omega \times \Omega} \frac{|u_s(x)-u_s(y)|^p}{|x-y|^{d+sp}} \d y\d x\Big)<\infty.
\end{align*}
If $\|u_s-u\|_{ L^p(\Omega)}\xrightarrow{s\to 1^-}0 $  for some $u\in L^p(\Omega)$, then the following assertions hold.
\begin{itemize}
\item 
If  $1<p<\infty$ then
$u\in W^{1,p}(\Omega)$ and we have
\begin{align*}
\|\nabla u\|^p_{L^p(\Omega)}
&\leq \liminf_{s\to1^-} \frac{C_{d,p,s}}{2}
\iint_{\Omega \times \Omega} \frac{|u_s(x)-u_s(y)|^p}{|x-y|^{d+sp}} \d y\d x\qquad 1<p<\infty.
\end{align*}
Likewise, if $p=1$ then  $u\in BV(\Omega)$ and its total variation $|u|_{BV(\Omega)}$ satisfies
\begin{align*}
|u|_{BV(\Omega)}
&\leq \liminf_{s\to1^-} \frac{C_{d,1,s}}{2}
\iint_{\Omega \times \Omega} \frac{|u_s(x)-u_s(y)|}{|x-y|^{d+s}} \d y\d x\qquad p=1.
\end{align*}

\item
If in addition, $\Omega$ is a robust $W^{s,p}$-extension domain then we have
\begin{align*}
\|u_s-u\|_{W^{\eta,p}(\Omega)} \xrightarrow{s\to 1^-}0\quad\text{for all $0\leq \eta<1$}.
\end{align*}
This remains true, for instance, in the following cases: $\Omega = \R^d$, $\Omega = \R^d_+$, $\partial\Omega$ is compact and Lipschitz, or $\Omega$ is the epigraph of a Lipschitz map.
\item 
If  in addition,  $\partial\Omega$  is compact and Lipschitz then for all $\tau\in (\frac{1}{p}, 1)$  we have
\begin{align*}
\|\gamma^s_0(u_s)-\gamma^1_0(u)\|_{L^{p}
(\partial \Omega)}+ \|\gamma^s_0(u_s)-\gamma^1_0(u)\|_{W^{\tau-\frac{1}{p},p}(\partial \Omega)} \xrightarrow{s\to 1^-}0.
\end{align*}
\end{itemize}
\end{theorem}

\begin{proof}
The  first claim can be found in \cite{Pon04} or in \cite[Theorem 5.37]{guy-thesis}.
If $\Omega$ is in particular a $W^{1,p}$-extension domain  then \cite[Lemma 2.12]{Fog23} we have
\begin{align}\label{eq:xxuniform-bound}
s(1-s)
\iint_{\Omega \times \Omega} \frac{|u(x)-u(y)|^p}{|x-y|^{d+sp}} \d y\d x\leq C(d,p,\Omega)
\begin{cases}
\|u\|^p_{W^{1,p}(\Omega)}&1<p<\infty\\
\|u\|_{BV(\Omega)} &p=1.
\end{cases}
\end{align}
This together with assumption on $u_s$ implies that $\|u_s-u\|^{\frac{\eta}{s}}_{W^{s,p}(\Omega)}\leq C$ with $C$ on depending on $d,p$ and $\Omega$. Thus, for the second claim, let us consider $0\leq \eta < s <1$  and put $\theta=\frac{\eta}{s}$ we have $u_s, u \in W^{\eta,p}(\Omega)$  and hence using Theorem \ref{thm:robust-interpo-dom} we get
\begin{align*}
\|u_s-u\|_{W^{\eta,p}(\Omega)}\leq C \|u_s-u\|^{\frac{\eta}{s}}_{W^{s,p}(\Omega)}
\|u_s-u\|^{1-\frac{\eta}{s}}_{L^{p}(\Omega)}\leq C \|u_s-u\|^{1-\frac{\eta}{s}}_{L^{p}(\Omega)}
\xrightarrow{s\to1^-}0.
\end{align*}

Note that $u_s, u \in W^{\tau,p}(\Omega)$ for $\frac{1}{p}<\tau\leq s<1$. According to the trace theorem (see Theorem \ref{thm:trace-frac-thm}) and the previous claims we  have
\begin{align*}
\|\gamma^s_0(u_s)-\gamma^1_0(u)\|_{L^p(\partial \Omega)}+ \|\gamma^s_0(u_s)-\gamma^1_0(u)\|_{W^{\tau-\frac{1}{p},p}(\partial \Omega)} \leq C_\tau \|u_s-u\|_{W^{\tau,p}( \Omega)}  \xrightarrow{s\to 1^-}0
\end{align*}
for some  a constant $C_\tau=C(d,p,\Omega,\tau)>0$. The third claim is proved.
\end{proof}

The following result establishes a form of asymptotic Lions--Peetre-type compactness (see \cite[Theorems~V.2.1 and~V.2.2]{LiPe64}), originally proved in \cite[Corollary~7]{BBM01}.

\begin{theorem}\label{thm:asymp-compact-frac}
Assume $\Omega \subset \R^d$ is open bounded and Lipschitz and let $s_0 \in [0,1)$. Assume $(u_s)_{s\in[0,1)} \subset L^p(\Omega)$ is a bounded family such that
\begin{align*}
\sup_{s\in(s_0,1)}
\Big(\int_\Omega |u_s(x)|^p\d x+ \frac{C_{d,p,s}}{2}
\iint_{\Omega \times \Omega} \frac{|u_s(x)-u_s(y)|^p}{|x-y|^{d+sp}} \d y\d x\Big)<\infty.
\end{align*}
Then there exist a subsequence $s_n\to 1^-$ and $u \in L^p(\Omega)$ such that
\begin{itemize}
\item $u_{s_n}\to u$ in $L^p(\Omega)$ and  for all $0\leq \eta<1$ we have
\begin{align*}
\|u_s-u\|_{W^{\eta,p}(\Omega)}
\xrightarrow{s\to 1^-}0\quad \text{for all $0\leq \eta<1$}.
\end{align*}
\item  
If $1<p<\infty$ then
$u\in W^{1,p}(\Omega)$ and we have
\begin{align*}
\|\nabla u\|^p_{L^p(\Omega)}
&\leq \liminf_{s_n\to 1^-} \frac{C_{d,p,s_n}}{2}
\iint_{\Omega \times \Omega} \frac{|u_{s_n}(x)-u_{s_n}(y)|^p}{|x-y|^{d+s_np}} \d y\d x\qquad 1<p<\infty
\end{align*}
while if  $p=1$ then $u\in BV(\Omega)$ and we have
\begin{align*}
|u|_{BV(\Omega)}
&\leq \liminf_{s_n\to 1^-} \frac{C_{d,1,s_n}}{2}
\iint_{\Omega \times \Omega} \frac{|u_{s_n}(x)-u_{s_n}(y)|}{|x-y|^{d+s_n}} \d y\d x\qquad p=1.
\end{align*}
\item If $1<p<\infty$ then for all $0< \tau <1$ and $\tau p>1$ we have
\begin{align*}
\|\gamma^{s_n}_0(u_{s_n})-\gamma^1_0(u)\|_{L^{p}(\partial \Omega)}+ \|\gamma_0^{s_n} (u_{s_n})-\gamma^1_0(u)\|_{W^{\tau-\frac{1}{p},p}(\partial \Omega)}
\xrightarrow{s_n\to 1^-}0.
\end{align*}
\end{itemize}
\end{theorem}
\begin{proof}
Following \cite{BBM01, Pon04}, we find a subsequence $s_n \to 1^-$ and $u \in L^p(\Omega)$ such that $u_{s_n} \to u$ in $L^p(\Omega)$.  The other claims follow directly from Theorem \ref{thm:asymp-converg-interpo}.
\end{proof}

\section{Asymptotics of the trace spaces} \label{sec:robust-trace}
In this section we investigate the convergence of the trace operator  and hence the convergence the fractional trace space $W^{s-\frac{1}{p},p}(\partial \Omega)$ of the fractional Sobolev space  $W^{s,p}(\Omega),$ $s\in (0,1)$ when $sp> 1$ and $\Omega$ is bounded Lipschitz.
We will need the following lemma.
\begin{lemma}
\label{lem:frac-unif-bound-trace-bis}
Let $\Omega \subset \R^d$ be an open set. Assume that $\Omega$ is a robust $W^{s,p}$-extension domain; this holds true, in particular, if $\Omega = \R^d$, $\Omega = \R^d_+$, or if $\partial\Omega$ is compact and Lipschitz. For  $1 \leq p < \infty$, there exists a constant $C = C(d,p,\Omega)>0$ such that:
\begin{enumerate}[$(i)$]
\item For any $0\leq  \eta < s \leq 1$ and all $u \in W^{s,p}(\Omega)$,
\begin{align*}
\|u\|_{W^{\eta,p}(\Omega)}
\leq C \|u\|_{W^{s,p}(\Omega)}.
\end{align*}

\item For any $\frac{1}{p} < \eta < s \leq 1$ and all $g \in W^{s-\frac{1}{p},p}(\partial \Omega)$,
\begin{align*}
\|g\|_{W^{\eta-\frac{1}{p},p}(\partial \Omega)} \leq C \|g\|_{W^{s-\frac{1}{p},p}(\partial \Omega)}.
\end{align*}
\end{enumerate}
\end{lemma}

\begin{proof}
Recall that $\|u\|_{W^{0,p}(\Omega)}= (1+\frac{|\mathbb{S}^{d-1}|}{p})^{1/p}\|u\|_{L^p(\Omega)}$. By Theorem \ref{thm:robust-interpo-dom} we get
\begin{align*}
\|u\|_{W^{\eta,p}(\Omega)} \leq C \|u\|_{W^{s,p}(\Omega)}^{\frac{\eta}{s}} \|u\|_{W^{0,p}(\Omega)}^{1-\frac{\eta}{s}}\leq C\|u\|_{W^{s,p}(\Omega)},
\end{align*}
for some  constant $C = C(d, p, \Omega)>0$. Let  $g\in W^{s-\frac{1}{p},p}(\partial \Omega)$ with $\frac{1}{p}<\eta<s\leq 1$  and consider $v\in W^{s,p}(\Omega)$ such that $\gamma^s_0(v)= g$. The previous estimate implies that
\begin{align*}
\|v\|_{W^{\eta,p}(\Omega)}\leq C\|v\|_{W^{s,p}(\Omega)}.
\end{align*}
Since $\gamma^\eta_0\mid_{W^{s,p}(\Omega)}= \gamma^s_0$ we obtain
\begin{align*}
\|g\|_{W^{\eta-\frac{1}{p},p}(\partial \Omega)}\leq C \|v\|_{W^{\eta,p}(\Omega)}\leq C\|v\|_{W^{s,p}(\Omega)}\leq C\|v\|^*_{W^{s,p}(\Omega)}.
\end{align*}
Passing to the infimum for $v\in W^{s,p}(\Omega)$ such that $\gamma^s_0 (v)=g$ we deduce that
\begin{align*}
\|g\|_{W^{\eta-\frac{1}{p},p}(\partial \Omega)}\leq C\|g\|_{W^{s-\frac{1}{p},p}(\partial \Omega)}.
\end{align*}
\end{proof}

\begin{theorem}
Assume $\Omega\subset \R^d$ is bounded Lipschitz.
There holds that
\begin{align*}
\lim_{s\to1^-}\|g\|_{W^{s-\frac{1}{p} ,p}(\partial \Omega)}= \|g\|_{W^{1-\frac{1}{p} ,p}(\partial \Omega)} \qquad \text{for all $g\in W^{1-\frac{1}{p} ,p}(\partial \Omega)$}.
\end{align*}
In particular, since $\gamma^s_0|_{W^{1,p}(\Omega)}=\gamma^1_0$, there holds that
\begin{align*}
\lim_{s\to1^-}\|\gamma^s_0(g)\|_{W^{s-\frac{1}{p} ,p}(\partial \Omega)}= \|\gamma^1_0(g)\|_{W^{1-\frac{1}{p} ,p}(\partial \Omega)}\qquad \text{for all $g\in W^{1,p}(\Omega)$}.
\end{align*}
\end{theorem}
\medskip

\begin{proof}
Let us first observe that, for $\frac{1}{p}<\eta\leq s\leq 1$, $\gamma^\eta_0$ and $\gamma^s_0$ both coincide on the dense subspace $C^1(\overline{\Omega})$ and hence,  $\gamma^\eta_0|_{W^{s,p}(\Omega)}=\gamma^s_0$, i.e., $\gamma^\eta_0(v)=\gamma^s_0(v)$ for all $v\in W^{s,p}(\Omega)$. Now, let us fix $u\in W^{1,p}(\Omega)\subset W^{s,p}(\Omega)$ such that  $\gamma^1_0(u)=\gamma^s_0(u)= g$ and hence
\begin{align*}
\|g\|_{W^{s-\frac{1}{p} ,p}(\partial \Omega)}  \leq  \|u\|^*_{W^{s,p}(\Omega)}.
\end{align*}
 Accordingly, since \eqref{eq:xasymp-form} implies $\| u\|^*_{W^{s,p}(\Omega)}\to \| u\|_{W^{1,p}(\Omega)}$ as $s\to 1^-$ we obtain
\begin{align*}
\limsup_{s\to1^-}\|g\|_{W^{s-\frac{1}{p} ,p}(\partial \Omega)}
&\leq \limsup_{s\to1^-}\|u\|^*_{W^{s,p}(\Omega)}
= \| u\|_{W^{1,p}(\Omega)}.
\end{align*}
Therefore, we deduce the limsup inequality
\begin{align}\label{eq:limsup-trace-loc}
\limsup_{s\to1^-}\|g\|_{W^{s-\frac{1}{p} ,p}(\partial \Omega)}  \leq
\|g\|_{W^{1-\frac{1}{p} ,p}(\partial\Omega)}.
\end{align}
For each $\frac{1}{p}<s<1$ there is  $u_s \in W^{s,p}(\Omega)$ such that $\gamma^s_0(u_s)=g$  and
\begin{align*}
\|u_s\|^*_{W^{s,p}(\Omega)}
\leq \|g\|_{W^{s-\frac{1}{p} ,p}(\partial \Omega)}  + 1-s.
\end{align*}
In view of Lemma \ref{lem:frac-unif-bound-trace-bis} there is constant $C=C(d,p,\Omega)>0$ such that
\begin{align}\label{eq:taylor-exapnsion}
\|g\|_{W^{s-\frac{1}{p} ,p}(\partial \Omega)}   \leq  C\|g\|_{W^{1-\frac{1}{p} ,p}(\partial \Omega)} .
\end{align}
It follows that $\gamma^{s}_0(u_s)= g$  and $(u_s)_s$ is asymptotically bounded, i.e.,
\begin{align*}
\sup_{s\in (\frac{1}{p},1)}\|u_s\|^*_{W^{s,p}(\Omega)}\leq 1+C\|g\|_{W^{1-\frac{1}{p} ,p}(\partial \Omega)} .
\end{align*}
By  Theorem \ref{thm:asymp-compact-frac}, there is   a subsequence $(s_n)_n$ with $s_n\to 1$ and $u\in W^{1,p}(\Omega)$ such that,  $\|u_{s_n}- u\|_{W^{\eta, p}(\Omega)}\xrightarrow{s_n\to1^-}0$ for all $0\leq \eta<1$ and
\begin{align*}
\|u\|_{W^{1,p}(\Omega)}
\leq \liminf_{s\to1^-}\|u_s\|^*_{W^{s,p}(\Omega)}
\leq \liminf_{s\to1^-}
\|g\|_{W^{s-\frac{1}{p} ,p}(\partial \Omega)}.
\end{align*}
For $\eta p>1$ and $\eta \leq s_n\leq 1$, we have  $\gamma^\eta_0(u_{s_n})=\gamma^{s_n}_0(u_{s_n})=g$ and $\gamma^{\eta}_0(u)=\gamma^1_0(u)$.
Thus we have $u_{s_n},u\in W^{\eta,p}(\partial\Omega)$,  $g-\gamma^1_0(u)= \gamma^\eta_0( u_{s_n})-\gamma^\eta_0(u)$. If follows that
\begin{align*}
\| g-\gamma^1_0(u)\|_{L^p(\partial \Omega)}
\leq C_\eta \| \gamma^{\eta}_0(u_{s_n}-u)\|_{W^{\eta-\frac{1}{p} ,p}(\partial \Omega)}
\leq \|u_{s_n}-u\|_{W^{\eta,p}(\Omega)}\xrightarrow{s_n\to 1^-}0.
\end{align*}
Thus, it appears that $u\in W^{1,p}(\Omega)$ and
$ \gamma^1_0(u)= g$. Accordingly we deduce that
\begin{align*}
\|g\|_{W^{1-\frac{1}{p} ,p}(\partial\Omega)}
\leq \|u\|_{W^{1,p}(\Omega)}
\leq \liminf_{s\to1^-} \|u_s\|^*_{W^{s,p}(\Omega)}
\leq  \liminf_{s\to1^-} \|g\|_{W^{s-\frac{1}{p} ,p}(\partial\Omega)}.
\end{align*}
Combining this with the limsup in \eqref{eq:limsup-trace-loc}  yields  the sought result
\begin{align*}
\lim_{s\to1^-}
\|g\|_{W^{s-\frac{1}{p} ,p}(\partial\Omega)}
= \|g\|_{W^{1-\frac{1}{p} ,p}(\partial\Omega)}.
\end{align*}
\end{proof}

In the next result we establish the asymptotic compactness for the trace spaces.
\begin{theorem}\label{thm:asymp-compact-frac-trace}
Assume $\Omega \subset \R^d$ is open bounded and Lipschitz, $1<p<\infty$ and let $s_0 \in [\frac{1}{p},1)$. Assume $(g_s)_{s\in[0,1)} \subset L^p(\partial \Omega)$ is a bounded family such that
\begin{align*}
\sup_{s\in(s_0,1)}
\|g_s\|_{W^{s-\frac{1}{p},p}(\partial\Omega)}<\infty.
\end{align*}
Then there exist a subsequence $s_n\to 1^-$ and $g \in W^{1-\frac{1}{p},p}(\partial\Omega)$ such that
\begin{align*}
\|g_{s_n}-g\|_{L^{p}(\partial \Omega)}+ \|g_{s_n}-g\|_{W^{\tau-\frac{1}{p},p}(\partial \Omega)}
\xrightarrow{s_n\to 1^-}0\quad \text{for all $\frac{1}{p}< \tau<1$}.
\end{align*}
Moreover, $g\in W^{1-\frac{1}{p},p}(\partial\Omega)$ and  we have
\begin{align*}
\|g\|_{W^{1-\frac{1}{p},p} (\partial\Omega)}
\leq \liminf_{s_n\to 1^-} \|g_{s_n}\|_{W^{s_n-\frac{1}{p},p}(\partial\Omega)}.
\end{align*}
\end{theorem}
\begin{proof}
For each $g_s$ there is $u_s\in W^{s,p}(\Omega)$ such that $\gamma^s_0 (u)_s= g_s$ and
\begin{align*}
\|u_s\|^*_{W^{s,p}(\Omega)}\leq \|g_s\|_{W^{s-\frac{1}{p}}(\partial\Omega)}+1-s.
\end{align*}
In particular, the family $(\|u_s\|_{W^{s,p}(\Omega)})_{s\in (s_0,1)}$ is uniformly bounded. According to Theorem \ref{thm:asymp-compact-frac} there is $u\in W^{1,p}(\Omega)$ and subsequence $s_n\to 1^-$ such that
\begin{align*}
\|\gamma^{s_n}_0(u_{s_n})-\gamma^1_0(u)\|_{L^{p}(\partial \Omega)}+ \|\gamma_0^{s_n} (u_{s_n})-\gamma^1_0(u)\|_{W^{\tau-\frac{1}{p},p}(\partial \Omega)}
\xrightarrow{s_n\to 1^-}0\quad \text{for all $\frac{1}{p}< \tau<1$}.
\end{align*}
In words since  $g_{s_n}= \gamma^{s_n}_0(u_{s_n)}$, letting   $g=\gamma^{1}_0(u)\in W^{1-\frac{1}{p},p}(\partial\Omega)$ we obtain
\begin{align*}
\|g_{s_n}-g\|_{L^{p}(\partial \Omega)}+ \|g_{s_n}-g\|_{W^{\tau-\frac{1}{p},p}(\partial \Omega)}
\xrightarrow{s_n\to 1^-}0\quad \text{for all $0\leq \tau<1$}.
\end{align*}
Moreover, Theorem \ref{thm:asymp-compact-frac} also implies
\begin{align*}
\|u\|_{W^{1,p}(\Omega)}
&\leq \liminf_{s_n\to 1^-} \|u_{s_n}\|^*_{W^{s_n,p}(\Omega)}\leq
\liminf_{s_n\to 1^-}
\|g_{s_n}\|_{W^{s_n-\frac{1}{p},p}(\partial \Omega)}.
\end{align*}
Therefore, by definition of the trace norm we find that
\begin{align*}
\|g\|_{W^{1-\frac{1}{p},p}(\partial \Omega)}=\|\gamma^1_0 (u)\|_{W^{1-\frac{1}{p},p}(\partial \Omega)}
\leq
\liminf_{s_n\to 1^-}
\|g_{s_n}\|_{W^{s_n-\frac{1}{p},p}(\partial \Omega)}.
\end{align*}
\end{proof}

As a consequence of Theorem \ref{thm:asymp-compact-frac-trace}, we obtain a characterization of $W^{1-\frac{1}{p},p}(\partial\Omega)$ in the spirit of Bourgain, Brezis, and Mironescu \cite{BBM01}.
\begin{theorem}[Characterization of $W^{1-\frac{1}{p},p}(\partial\Omega)$]
\label{thm:charac-frac-trace}
Assume $\Omega \subset \R^d$ is open bounded and Lipschitz. Assume $g\in L^p(\partial \Omega)$, $1<p<\infty$ satisfies
\begin{align*}
\liminf_{s\to 1^-}
\|g\|_{W^{s-\frac{1}{p},p}(\partial\Omega)}<\infty.
\end{align*}
Then we have $g \in W^{1-\frac{1}{p},p}(\partial\Omega)$ and
\begin{align*}
\|g\|_{W^{1-\frac{1}{p},p} (\partial\Omega)}
\leq \liminf_{s\to 1^-} \|g\|_{W^{s-\frac{1}{p},p}(\partial\Omega)}.
\end{align*}
\end{theorem}
\begin{proof}
It is sufficient to apply Theorem \ref{thm:asymp-compact-frac-trace} to the constant sequence $g_s=g$.
\end{proof}

\section{General robust Poincar\'e inequalities }\label{sec:gene-fract-poin-ineq}
In this section we deal with various robust Poincar\'e inequalities. We begin with the nonlocal version. A closely related inequality is established with non-robust constant in \cite[Theorem 2.2]{RaZi23} in the setting of Triebel–Lizorkin spaces.

\begin{theorem}[General robust Poincar\'e-Friedrichs]
Given any open bounded  set $\Omega\subset \R^d$  and $1\leq p< \infty$, there is a constant $C=C(d,p,\Omega)>1$ such that
\begin{align*}
[u]_{W^{\eta,p}(\R^d)}\leq C[u]_{W^{s,p}(\R^d)},
\end{align*}
for all $u\in C^\infty_c(\Omega)$  and $0\leq \eta <s \leq 1$.
In particular for $\eta=0$ and $s=1$ yields
\begin{align*}
\|u\|_{L^p(\Omega)}\leq C\|\nabla u\|_{L^p(\Omega)}.
\end{align*}
\end{theorem}

\begin{proof}
By \cite[Theorem 10.8]{Fog25} there is $B=B(d,p,\Omega)>1$  such that
\begin{align*}
[u]_{W^{0,p}(\R^d)}\leq B[u]_{W^{s,p}(\R^d)}.
\end{align*}
The latter also provides the inequality in the case $\eta=0$. This together with  the robust interpolation inequality \cite[Theorem 1.1]{Fog26rob} imply that
\begin{align*}
[u]_{W^{\eta,p}(\R^d)}\leq 8[u]^{\frac{\eta}{s}}_{W^{s,p}(\R^d)}[u]^{(1-\frac{\eta}{s})}_{W^{0,p}(\R^d)}
\leq C[u]_{W^{s,p}(\R^d)}.
\end{align*}
The case $s=1$ follows immediately by letting $s\to 1^-$.
\end{proof}

\begin{theorem}[General robust Poincar\'e]
Assume the open set $\Omega\subset \R^d$ is  bounded Lipschitz and connected. Let  $1\leq p< \infty$. Then there is a constant $C=C(d,p,\Omega)>1$ such that for all $u\in L^p(\Omega)$  and $0\leq \eta <s \leq 1$  we have
\begin{align*}
\|u-\mbox{$\fint_\Omega$} u \|_{W^{\eta,p}(\Omega)}\leq s^{-\frac{1}{sp}} \, C[u]_{W^{s,p}(\Omega)}.
\end{align*}
In particular for $\eta=0$ and $s=1$ yields
\begin{align*}
\|u-\mbox{$\fint_\Omega$} u\|_{L^p(\Omega)}\leq C\|\nabla u\|_{L^p(\Omega)}.
\end{align*}
\end{theorem}
\begin{remark}
It is worth noting that the above estimate remains consistent as $s \to 0^+$ for $u \in \bigcup_{0 < s < 1} W^{s,p}(\Omega)$, in the sense that no blow-up occurs as $s \to 0^+$. Indeed, letting $s \to 0^+$, we have $s^{1/s} \to 0$. Moreover, since $\Omega$ is a bounded Lipschitz domain and $0 \leq \eta < s$, \cite[Theorem 4.6]{Fog26rob} implies that $[u]_{W^{s,p}(\Omega)} \to 0$ and $\|u-\mbox{$\fint_\Omega$}u\|_{W^{\eta,p}(\Omega)}\to \|u-\mbox{$\fint_\Omega$}u\|_{L^p(\Omega)}$ as $s\to 0^+$.
\end{remark}

\begin{proof}
According to Theorem \ref{thm:robust-interpo-dom} we have
\begin{align*}
\| u-\mbox{$\fint_\Omega$}u \|_{W^{\eta,p}(\R^d)}\leq  C\|u-\mbox{$\fint_\Omega$}u\|^{\frac{\eta}{s}}_{W^{s,p}(\Omega)}\|u-\mbox{$\fint_\Omega$}u\|^{1-\frac{\eta}{s}}_{L^p(\Omega)}.
\end{align*}
By \cite[Corollary 10.4]{Fog25} there is $B=B(d,p,\Omega)>1$ independent on $s$ such that
\begin{align*}
\|u-\mbox{$\fint_\Omega$}u\|_{L^p(\Omega)}\leq B\Big((1-s)
\iint_{\Omega\times\Omega}\frac{|u(x)-u(y)|^p}{|x-y|^{d+sp}}  \d y \d x\Big)^{1/p}= s^{-1/p} B[u]_{W^{s,p}(\Omega)}.
\end{align*}
Inserting this in the previous  estimate gives
\begin{align*}
\|u-\mbox{$\fint_\Omega$}u\|_{W^{\eta,p}(\Omega)}\leq 2^{\frac{\eta}{sp} } s^{-\frac{1}{sp}} CB[u]_{W^{s,p}(\Omega)}
\leq s^{-\frac{1}{sp}} C[u]_{W^{s,p}(\Omega)},
\end{align*}
where we used the fact that, $[u-\mbox{$\fint_\Omega$}u]_{W^{s,p}(\Omega)}= [u]_{W^{s,p}(\Omega)}$, $\eta/s <1$,  $s^{-1/s}>1$ and $B>1$.  Finally since $\Omega$ is bounded Lipschitz, we have $[u]_{W^{s,p}(\Omega)}\to [u]_{W^{1,p}(\Omega)}$ as $s\to 1^-$ while  by \cite[Theorem 4.6]{Fog26rob} letting $\eta\to 0^+$ implies $\|u-\mbox{$\fint_\Omega$}u\|_{W^{\eta,p}(\Omega)}\to \|u-\mbox{$\fint_\Omega$}u\|_{L^p(\Omega)}$.  Thus we obtain the extreme cases $\eta=0$ and/or $s=1$.
\end{proof}

\begin{theorem}[General robust regional Poincar\'e–Friedrichs]
\label{thm:robust-poinca-fried-regio}
Assume  $\Omega\subset \R^d$ is open bounded with Lipschitz boundary and $1< p <\infty$. Then there exist $s_0= s_0(d,p,\Omega)\in (\frac{1}{p},1 )$  and $C=C(d,p,\Omega)>1$ such that for all $u\in W^{s,p}_0(\Omega)$  and $s_0 \leq s <1$
\begin{align}\label{eq:rob-poincare-fried-regio}
\int_\Omega |u(x)|^p\d x\leq Bs(1-s) \iint_{ \Omega \times \Omega}\frac{|u(x)-u(y)|^p}{|x-y|^{d+sp}}\d y \d x.
\end{align}
Furthermore, for all $s$ and $\eta$ satisfying $0 \leq \max\{s_0, \eta\} < s \leq 1$,
\begin{align}\label{eq:rob-gene-poincare}
\|u\|_{W^{\eta,p}(\Omega)} \leq C [u]_{W^{s,p}(\Omega)} \quad \text{for all } u \in W^{s,p}_0(\Omega).
\end{align}
In particular, setting $\eta = 0$ and $s = 1$ yields the classical Poincar\'e inequality
\begin{align*}
\|u\|_{L^p(\Omega)} \leq C \|\nabla u\|_{L^p(\Omega)}.
\end{align*}
\end{theorem}

\begin{proof}
Assume there exist no such $s_0$ and $B$. For $n \geq 1$ large enough, taking $s_0 = 1-\frac{1}{2^n}>\frac{1}{p}$ and $ = 2^n$, there exist $s_n \in (1-\frac{1}{2^n},1)$ and $u_{s_n} \in W^{s_n, p}_0(\Omega)$ such that
\begin{align*}
\|u_{s_n}\|_{L^p(\Omega)} = 1\quad\text{and}\quad s_n(1-s_n)\iint_{ \Omega \times \Omega}\frac{|u_{s_n}(x)-u_{s_n}(y)|^p}{|x-y|^{d+s_np}}\d y \d x<\frac{1}{2^n}.
\end{align*}
In particular, $\gamma^{s_n}_0 (u_{s_n})=0$ as $u_{s_n}\in W^{s_n, p}_0(\Omega)$,  and
\begin{align*}
\sup_{n\geq 1}\Big( \int_\Omega |u_{s_n}(x)|^p\d x+ s_n(1-s_n)\iint_{ \Omega \times \Omega}\frac{|u_{s_n}(x)-u_{s_n}(y)|^p}{|x-y|^{d+s_np}}\d y \d x\Big)<2.
\end{align*}
According to Theorem \ref{thm:asymp-compact-frac} there is $u \in W^{1,p}(\Omega)$ and a subsequence still denoted $(u_{s_n})_n$ converging to $u$ in $L^p(\Omega)$ . Moreover,  since $\gamma^{s_n}_0(u_{s_n})=0$ we have
\begin{align*}
&K_{d,p} \frac{|\mathbb{S}^{d-1}|}{p}\|\nabla u\|_{L^p(\Omega)}^p \leq \liminf_{n\to\infty} s_n(1-s_n)\iint_{ \Omega \times \Omega}\frac{|u_{s_n}(x)-u_{s_n}(y)|^p}{|x-y|^{d+s_np}}\d y \d x = 0,\\
&\|\gamma^1_0(u)\|_{W^{\tau-\frac{1}{p},p}(\partial\Omega)}
=\|\gamma^{s_n}_0(u_{s_n}) -\gamma^1_0(u)\|_{W^{\tau-\frac{1}{p},p}(\partial\Omega)}\xrightarrow{s_n\to 1^-}0\qquad\text{for all $\frac1{p} <\tau<1$}.
\end{align*}
It follows that  $\nabla u = 0$ a.e. in $\Omega$ and $\gamma^1_0 (u)=0 $. Thence, since  $\partial\Omega$ is Lipschitz, we deduce that  $u \in W^{1,p}_0(\Omega)$ and $\|u\|_{L^p(\Omega)}=1$. Consequently,  the zero extension $\widetilde{u}$ of $u$ belongs to  $ W^{1,p}(\R^d)$ and satisfies $\nabla \widetilde{u}=  \widetilde{\nabla u}= 0$ a.e. in $\R^d$. It follows that $\widetilde{u}$ is constant on $\R^d$. We deduce that $\widetilde{u}\equiv0$ on $\R^d$, since $\widetilde{u}= 0$ on $\R^d \setminus \Omega$. This clearly implies that  $u \equiv 0$ on $\Omega$. This goes against  the fact that $\|u\|_{L^p(\Omega)}=1$  and thus our initial assumption was wrong. Next, the inequality \eqref{eq:rob-gene-poincare} is merely  obtained  by combining  the robust interpolation from Theorem \ref{thm:robust-interpo-dom}  and the case the case $\eta = 0$ yielded by Theorem \ref{thm:robust-poinca-fried-regio}, as follows
\begin{align*}
\| u\|_{W^{\eta,p}(\R^d)}
&\leq  C\|u\|^{\frac{\eta}{s}}_{W^{s,p}(\Omega)}\|u\|^{1-\frac{\eta}{s}}_{L^p(\Omega)}\qquad(0\leq \eta <s \leq 1)\\
&\leq CB[u]_{W^{s,p}(\Omega)}\qquad \quad (0\leq \max(s_0,\eta) <s \leq 1),
\end{align*}
where  $s_0\in( \frac{1}{p},1)$ and $B>1$ only depend on $d,p$ and $\Omega$.
\end{proof}

\begin{theorem}[Robust Friedrichs Inequality I]
\label{thm:rob-friedrichs}
Let $\Omega \subset \R^d$ be a bounded, connected, open set with a compact Lipschitz boundary $\partial\Omega$, and let $1 < p < \infty$. Let $\Gamma_0 \subset \partial\Omega$ be a subset of the boundary with positive $(d-1)$-dimensional Hausdorff measure $\mathcal{H}^{d-1}(\Gamma_0)>0$. Then there exist constants $s_0 = s_0(d,p,\Omega, \Gamma_0) \in (\frac{1}{p}, 1)$ and $C = C(d,p,\Omega, \Gamma_0) > 0$ such that, for all $s \in (s_0, 1)$ and all $u \in W^{s,p}(\Omega)$,
\begin{align}\label{eq:robust_friedrichs_1}
\int_\Omega |u(x)|^p \d x \leq C s(1-s) \iint_{\Omega \times \Omega} \frac{|u(x)-u(y)|^p}{|x-y|^{d+sp}} \d y \d x + C \Big| \int_{\Gamma_0} \gamma^s_0 (u)\d\sigma  \Big|^p.
\end{align}
Equivalently we have that, for all $s \in (s_0, 1)$ and all $u \in W^{s,p}(\Omega)$,
\begin{align}\label{eq:robust_friedrichs_2}
\int_\Omega \Big| u(x) - \tfrac{1}{\mathcal{H}^{d-1}(\Gamma_0)}\int_{\Gamma_0} \gamma^s_0 (u) \d\sigma  \Big|^p \d x \leq C (1-s) \iint_{\Omega \times \Omega} \frac{|u(x)-u(y)|^p}{|x-y|^{d+sp}} \d y \d x.
\end{align}
Furthermore, for all $s$ and $\eta$ satisfying $0 \leq \max\{s_0, \eta\} < s \leq 1$, we have
\begin{align}\label{eq:generalized_friedrichs-bis}
\|u -\mbox{$\tfrac{1}{\mathcal{H}^{d-1}(\Gamma_0)}\int_{\Gamma_0}$} \gamma^s_0(u)\d\sigma \|_{W^{\eta,p}(\Omega)} \leq C [u]_{W^{s,p}(\Omega)} \quad \text{for all } u \in W^{s,p}(\Omega).
\end{align}
\end{theorem}
\begin{proof}
We proceed by contradiction. Assume that no such constants $s_0$ and $C$ exist. Then, for each integer $n \geq 1$ sufficiently large such that $s_0 =1 - \frac{1}{2^n} > \frac{1}{p}$ and $C_n =2^n$, there exist $s_n \in (1 - \frac{1}{2^n}, 1)$ and $u_{s_n} \in W^{s_n, p}(\Omega)$ satisfying $\|u_{s_n}\|_{L^p(\Omega)} = 1$ and
\begin{align*}
s_n(1-s_n) \iint_{\Omega \times \Omega} \frac{|u_{s_n}(x)-u_{s_n}(y)|^p}{|x-y|^{d+s_n p}} \d y \d x + \Big| \int_{\Gamma_0} \gamma^{s_n}_0 (u_{s_n}) \d\sigma  \Big|^p < \frac{1}{2^n}.
\end{align*}
Consequently, the sequence $(u_{s_n})_{n}$ satisfies the uniform bound
\begin{align*}
\sup_{n \geq 1} \left( \int_\Omega |u_{s_n}(x)|^p \d x + s_n(1-s_n) \iint_{\Omega \times \Omega} \frac{|u_{s_n}(x)-u_{s_n}(y)|^p}{|x-y|^{d+s_n p}} \d y \d x \right) < 2.
\end{align*}
According to Theorem \ref{thm:asymp-compact-frac}, there exist a function $u \in W^{1,p}(\Omega)$ and a subsequence still denoted by $(u_{s_n})_n$, such that $u_{s_n} \to u$ strongly in $W^{\eta, p}(\Omega),$ $0\leq \eta<1$  and
\begin{align*}
&K_{d,p} \frac{|\mathbb{S}^{d-1}|}{p} \|\nabla u\|_{L^p(\Omega)}^p \leq \liminf_{n \to \infty} s_n(1-s_n) \iint_{\Omega \times \Omega} \frac{|u_{s_n}(x)-u_{s_n}(y)|^p}{|x-y|^{d+s_n p}} \d y \d x = 0,\\
&\|\gamma^{s_n}_0(u)-\gamma^1_0(u)\|_{L^{p}(\partial\Omega)} + \|u_{s_n}-u\|_{L^p(\partial\Omega)} \xrightarrow{n \to \infty} 0 .
\end{align*}
We immediately find that $\|u\|_{L^p(\Omega)} = 1$ and $u \equiv c$ on $\Omega$ for some $c \in \mathbb{R}$ since $\Omega$ is connected and $\nabla u=0$ a.e. on $\Omega$. Moreover, the convergence of the traces  implies
\begin{align*}
\Big|\int_{\Gamma_0} \gamma^1_0 (u) \d\sigma  \Big|\leq \Big|  \int_{\Gamma_0} \gamma^{s_n}_0 (u_{s_n})\d\sigma  \Big|+|\partial\Omega|^{1-\frac{1}{p}}\|\gamma^{s_n}_0(u)-\gamma^1_0(u)\|_{L^{p}(\partial\Omega)}  \xrightarrow{n\to \infty} 0.
\end{align*}
Given that $u\equiv c$ on  $\Omega$ and hence $\gamma^1_0 (u) \equiv c$ on $\partial\Omega$, it follows that
\begin{align*}
0 = \int_{\Gamma_0} \gamma^1_0( u)\d\sigma  = \int_{\Gamma_0} c \d\sigma  = c  \mathcal{H}^{d-1}(\Gamma_0).
\end{align*}
Since  $\mathcal{H}^{d-1}(\Gamma_0)> 0$, we must have $c = 0$. Thus $u \equiv 0$ on $\Omega$. This contradicts the fact that $\|u\|_{L^p(\Omega)} = 1$. Thus, the initial assumption is false, completing the proof.
\end{proof}

The following Theorem is an immediate consequence of Theorem \ref{thm:rob-friedrichs}.
\begin{theorem}[Robust Friedrichs Inequality II]
\label{thm:rob-friedrichs-zero}
Let $\Omega \subset \R^d$ be open bounded Lipschitz and connected and let $1 < p < \infty$. Let $\Gamma_0 \subset \partial\Omega$ be a subset of the boundary with positive $(d-1)$-dimensional Hausdorff measure $\mathcal{H}^{d-1}(\Gamma_0)>0$.
Then there exist a constant $C = C(d,p,\Omega, \Gamma_0) > 0$  and $s_0= s_0(d,p,\Omega, \Gamma_0) \in (\frac{1}{p}, 1)$ such that,
\begin{align}\label{eq:robust_friedrichs_vanishing}
\int_\Omega |u(x)|^p \d x \leq C s(1-s)\iint_{\Omega \times \Omega} \frac{|u(x)-u(y)|^p}{|x-y|^{d+sp}} \d y \d x,
\end{align}
for all $s \in (s_0, 1)$ and all $u \in W^{s,p}_{\Gamma_0}(\Omega)$, where we define
\begin{align*}
W^{s,p}_{\Gamma_0}(\Omega)
:= \big\{ u \in W^{s,p}(\Omega) : \gamma^s_0 (u) = 0 \text{ a.e. on } \Gamma_0 \big\}.
\end{align*}
Furthermore, for all $s$ and $\eta$ satisfying $0 \leq \max\{s_0, \eta\} < s \leq 1$, we have
\begin{align}\label{eq:generalized_friedrichs}
\|u\|_{W^{\eta,p}(\Omega)} \leq C [u]_{W^{s,p}(\Omega)} \quad \text{for all } u \in W^{s,p}_{\Gamma_0}(\Omega).
\end{align}
\end{theorem}

\section{Proofs of the main results}\label{sec:opti-conv-regio-frac}
In this section, we show the convergence of weak solutions to the Dirichlet problem driven by the normalized regional fractional $p$-Laplacian $(-\Delta)_{p,\Omega}^s$, for $0 < s \leq 1$.
\subsection{Well-posedness of weak solutions }  \label{subsec:regional-frac-welposedness}
For $sp>1$, $f\in (W^{s,p}_0(\Omega))'$ and $g\in W^{s-\frac{1}{p}, p}(\partial \Omega)$, we establish the well-posedness of the Dirichlet  problem
\begin{align}\label{eq:regional-dirichlet}\tag{$D_s$}
(-\Delta)^s_{p,\Omega} u = f \quad\text{in}~~ \Omega \quad\quad\text{ and } \quad\quad  \gamma^s_0(u)=g ~~~ \text{on}~~ \partial\Omega.
\end{align}
A formal heuristic calculation reveals that for $u\in W^{s,p}(\Omega)$  and $v\in C^\infty_c(\Omega)$ we have
\begin{align*}
\int_\Omega (-\Delta)^s_{p,\Omega} u \, v \d x = \cE^{s,p}_\Omega(u,v),
\end{align*}
where we recall that  the form  $\cE^{s,p}_\Omega(\cdot,\cdot)$ associated to $(-\Delta)_{p,\Omega}^s$ are given by
\begin{align*}
\cE^{s,p}_{\Omega}(u,v) &:= \frac{C_{d,p,s}}{2} \iint_{\Omega \times \Omega}  \frac{|u(y)-u(x)|^{p-2}(u(y)-u(x)) }{|x-y|^{d+sp}} (v(y)-v(x))\d y \, \d x,\\
\mathcal{E}^{1,p}_{\Omega}(u,v)&:= \int_\Omega |\nabla u(x)|^{p-2}\nabla u(x)\cdot \nabla v(x)\d x.
\end{align*}
Thus, we can formally  identify $(-\Delta)^s_{p,\Omega} u\equiv \cE^{s,p}_\Omega(u,\cdot)$.
From this is legitimate to say that   $u \in W^{s,p}(\Omega)$ is a weak solution of the Dirichlet problem \eqref{eq:regional-dirichlet} if
\begin{align}
\label{eq:var-regio-frac-dirichlet}\tag{$V_s$}
u-g\in W^{s,p}_0(\Omega)\quad \text{and}\quad \cE^{s,p}_\Omega(u,v) = \langle f , v \rangle_s   \quad \mbox{for all}~~v \in W^{s,p}_0(\Omega)\,.
\end{align}
Actually, by a mere adaptation of \cite[Proposition 8.17]{Fog25} one arrives at the conclusion that, for any  $\overline{g} \in W^{s,p}(\Omega)$ with, $\gamma^s_0(\overline{g}) = g$, a function  $u \in W^{s,p}(\Omega)$ is a solution to the variational problem \eqref{eq:var-regio-frac-dirichlet} if and only if $u$ verifies the minimization
\begin{align}\label{eq:min-regio-frac-dirichlet}
\tag{$M_s$}
\frac{1}{p} \cE^{s,p}_\Omega (u,u) - \langle f , u-\overline{g}\rangle_s= \min_{v \in \overline{g}+ W^{s,p}_0(\Omega)} \Big\{ \frac{1}{p} \cE^{s,p}_\Omega (v,v) - \langle f , v-\overline{g}\rangle_s \Big\}.
\end{align}
Moreover,  the  energy functional $v\mapsto \frac{1}{p} \cE^{s,p}_\Omega (v,v) - \langle f , v-\overline{g}\rangle_s$ is strictly convex, and its coercivity is guaranteed by the Friedrichs–Poincaré inequality \eqref{eq:rob-poincare-fried-regio}.  These properties ensure the existence and uniqueness of a unique minimizer $u$. In particular, $u$ is independent of the specific choice of the extension $\overline{g}$. Consequently, one can  obtain the following well-posedness result by adapting \cite[Theorem 8.18]{Fog25}.
\begin{theorem}
\label{thm:nonlocal-dirichlet-gen}
Assume $\Omega\subset \R^d$  is bounded Lipschitz and $sp>1$. Let $f\in (W^{s,p}_0(\Omega))'$ and $g\in W^{s-\frac{1}{p},p}(\partial \Omega)$.   There is $C= C(d,p,\Omega)>0$ independent of $s$ such that, the following assertions hold.
\begin{enumerate}[$(i)$]
\item \textbf{Existence}. The variational problem \eqref{eq:var-regio-frac-dirichlet} has a unique  solution $u\in W^{s,p}(\Omega)$.
\item \textbf{Boundedness}. For any $\overline{g}\in W^{s,p}(\Omega)$ with $\gamma^s_0(\overline{g})=g$ we have
\begin{align}\label{eq:energy-bound-D}
\cE^{s,p}_\Omega(u,u)
&\leq C(\|f\|_{(W^{s,p}_0(\Omega))'} ^{p'} + \cE^{s,p}_\Omega(\overline{g},\overline{g})),
\\
\label{eq:sol-bd-regio-frac-D}
\|u \|_{W^{s,p}(\Omega)}&\leq C \big(\|f\|^{p'}_{(W^{s,p}_0(\Omega))'}+\|g\|^p_{ W^{s-\frac{1}{p},p}(\partial \Omega)}\big)^{1/p}.
\end{align}
\item \textbf{Continuity}.
Let $u_i$ be the solution associated with $f=f_i$ and $g=g_i$, $i=1,2$. Let us put
$D=D(f_1,f_2,g_1,g_2)=\sum_{i=1}^2  \big(\|f_i\|^{p'}_{(W^{s,p}_0(\Omega))'}+\|g_i\|^p_{ W^{s-\frac{1}{p},p}(\partial \Omega)}\big)^{1/p'}$
\end{enumerate}
If $p\geq 2$ we have
\begin{align*}
\|u_1-u_2\|_{W^{s,p}(\Omega)}
\leq C\Big(\|f_1-&f_2\|^{1/(p-1)}_{(W^{s,p}_0(\Omega))'} +\|g_1-g_2\|_{ W^{s-\frac{1}{p},p}(\partial \Omega)}\hspace{-1ex}\\
&+  D^{1/p}\|g_1-g_2\|^{1/p}_{ W^{s-\frac{1}{p},p}(\partial \Omega)}\Big).
\end{align*}
If $1<p<2$ , then for $R=	(D^{\frac{1}{p-1}} + D)^{1/2}$ we have
\begin{align*}
\|u_1-u_2\|_{W^{s,p}(\Omega)} \leq C \Big(D^{\frac{2-p}{p-1}} &\|f_1-f_2\|_{(W^{s,p}_0(\Omega))'}+ \|g_1-g_2\|_{ W^{s-\frac{1}{p},p}(\partial \Omega)}\hspace{-1.4ex}\\
&+
R\|g_1-g_2\|_{ W^{s-\frac{1}{p},p}(\partial \Omega)}^{1/2}\Big).
\end{align*}
\end{theorem}

\begin{theorem}[Weak comparison principle]\label{thm:compa-princ-regio}
Let $\Omega \subset \R^d$ be a bounded Lipschitz domain, $sp>1$ and let $u, v \in W^{s,p}(\Omega)$. Assume that $v \leq u$ on $\partial \Omega$ and $(-\Delta)^s_{p,\Omega}v \leq (-\Delta)^s_{p,\Omega} u$ in $\Omega$ in the weak sense; that is,
\begin{align*}
\cE^{s,p}_\Omega (v, w) \leq \cE^{s,p}_\Omega (u,w) \qquad \text{for all } w \in W^{s,p}_0(\Omega) \text{ with } w \geq 0.
\end{align*}
Then we have $v \leq u$ a.e.\ in $\Omega$.
\end{theorem}
\begin{proof}
For $t\in \R$ we put $\psi(t)=|t|^{p-2}t$ and $t_{\pm}=\max(\pm t, 0)$ so that $t=t_{+}-t_{-}$. By \cite[Corollary A.7.]{Fog25}, for some constant $A'_p>0$ and for all $a_1, a_2, b_1, b_2 \in \R$ we have the estimates
\begin{align*}
(\psi(b_1-b_2)- \psi(a_1-a_2)) &((b_1-a_1)_{+}-(b_2-a_2)_{+})\geq
\\
&\begin{cases}
A'_p |(b_1-a_1)_{+}-(b_2-a_2)_{+}|^p& \hspace{-1ex} p\geq2,
\\
A'_p |(b_1-a_1)_{+}-(b_2-a_2)_{+}|^2(|b|+|a|)^{p-2} &\hspace{-1ex}  1< p<2.
\end{cases}
\end{align*}
Let $w= (v-u)_+$ so that $w=0$ on $\partial \Omega$ since $v-u\leq0$ on $\partial \Omega$. Hence $w\in W^{s,p}_0(\Omega)$ and $w\geq0$. Taking $b_1= v(x), \, b_2= v(y), \,a_1= u(x),\, a_2=u(y)$ and using H\"older inequality we find that
\begin{align*}
0\geq \cE^{s,p}_\Omega (v, w)-\cE^{s,p}_\Omega (u, w) &\geq A'_p\cE^{s,p}_\Omega  (w,w)\qquad \text{$p\geq2$},
\\
0\geq \big(\cE^{s,p}_\Omega (v,w)-\cE^{s,p}_\Omega (u,w)\big)
\big( \cE^{s,p}_\Omega (v,v)+\cE^{s,p}_\Omega (u,u)\big)^{\frac{2-p}{p}} &\geq c_p \cE^{s,p}_\Omega (w,w)^{\frac{2}{p}}\qquad\text{$1<p<2$}.
\end{align*}
In both cases we get $\cE^{s,p}_\Omega (w,w)=0$. The Poincar\'{e}-Friedrichs inequality (Theorem \ref{thm:rob-friedrichs-zero}) implies  $\|w\|_{L^p(\Omega)}=0$. Hence  $w=(v-u)_+=0$ a.e on $\Omega$ that is  $v\leq u$ a.e. on $\Omega$.
\end{proof}

\begin{theorem}[Weak  and strong maximum principles]
Let $\Omega \subset \R^d$ be an open, bounded Lipschitz domain, and let $u \in W^{s,p}(\Omega)$. Assume that $u \geq 0$ a.e. on $\partial\Omega$ and $(-\Delta)^s_{p,\Omega}u \geq 0$ in $\Omega$ in the weak sense. Then, the following hold:
\begin{itemize}
    \item \textit{Weak Maximum Principle:} The solution satisfies $u \geq 0$ a.e. in $\Omega$.
    \item \textit{Strong Maximum Principle:} If $u$ is additionally continuous on $\Omega$, then either $u \equiv 0$ in $\Omega$ or $u > 0$ in $\Omega$.
\end{itemize}
\end{theorem}

\begin{proof} The fact that $u\geq 0$ a.e. in $\Omega$ follows from Theorem \ref{thm:compa-princ-regio}. Assume $u$ is continuous. If $u(x_0)\leq 0$ for some $x_0\in \Omega$ then  since $\psi(t)= |t|^{p-2}t$ is increasing we have
\begin{align*}
0\leq (-\Delta)^s_{p,\Omega}u(x_0)=\pv C_{d,p,s} \int_{\Omega} \frac{\psi(u(x_0)-u(y))}{|x_0-y|^{d+sp}}\d y\leq -C_{d,p,s} \int_{\Omega} \frac{(u(y))^{p-1}}{|x_0-y|^{d+sp}}\d y\leq 0,
\end{align*}
which is only possible if $u (y)=0$ for all $y\in\Omega $ since $u\geq0$.
\end{proof}

\subsection{Gamma convergence} Let us first establish the $\Gamma$-convergence of the functionals associated with the Dirichlet problems under consideration.
\begin{theorem}
\label{thm:gamma-conv-regio-Jo}
Let $\Omega \subset \R^d$ be a bounded Lipschitz domain and $1 < p < \infty$. Assume $(f_s)_s$ with   $f_s\in  (W^{s,p}_{0}(\Omega))'$ converges asymptotically weakly to $f_1\in (W^{1,p}_0(\Omega))'$ and $ (g_s)_s$ with  $g_s\in  W^{s,p}(\Omega)$ converges asymptotically strongly to $g_1\in W^{1,p}(\Omega)$, that is,
\begin{align*}
\lim_{s\to1^-}	\|g_s-g_1\|_{W^{s,p}(\Omega)}=0.
\end{align*}
We extend the functional $\cJ^{s,p}_0: L^p(\Omega) \to (-\infty, \infty]$ by
\begin{align*}
\cJ^{s,p}_0(v) =
\begin{cases}
\frac{1}{p} \cE^{s,p}_\Omega(v,v) - \langle f_s, v-g_s \rangle_s & \text{if } v \in g_s + W^{s,p}_0(\Omega), \\
\infty & \text{otherwise},
\end{cases}
\end{align*}
Then, for any $0 \leq \eta < 1$, $(\cJ^{s,p}_0)_{s \in (0,1)}$ $\Gamma$-converges to $\mathcal{J}^{1,p}_0$ in the strong topology of $W^{\eta, p}(\Omega)$ as $s \to 1^-$. To be more precise, the following two assertions hold:

\begin{itemize}
\item \textbf{Limsup: 
} For every  $v\in W^{\eta,p}(\Omega)$ there exists a sequence $(v_s)_s$ such that  $v_s\to v$ in $W^{\eta,p}(\Omega)$ as $s \to 1^-$ and
\begin{align*}
\limsup_{s \to 1^-} \cJ^{s,p}_0(v_s) \leq \mathcal{J}^{1,p}_0(v).
\end{align*}
\item \textbf{Liminf:} For every $v \in L^p(\Omega)$ and every sequence $(v_s)_s \subset L^p(\Omega)$ satisfying $v_s \to v$ strongly in $W^{\eta,p}(\Omega)$ as $s \to 1^-$, we have
\begin{align*}
\liminf_{s \to 1^-} \cJ^{s,p}_0(v_s) \geq \mathcal{J}^{1,p}_0(v).
\end{align*}
\end{itemize}
\end{theorem}

\begin{proof}
Let $v\in W^{\eta, p}(\Omega)$.  If  $v\in L^p(\Omega)\setminus (g_1+W^{1,p}_0(\Omega))$ then $\cJ^{1,p}_0(v)=\infty$ and the limsup holds.  In this case, it is sufficient to consider  the constant sequence $v_s=v$.
Now assume $v\in g_1+ W^{1,p}_0(\Omega)\subset g_1+W^{s,p}_0(\Omega)$ and put $v_s=g_s+v -g_1 $
then $v_s\in g_s+ W^{s,p}_0(\Omega)$ and  $v_s\to v$ in $W^{\eta,p}(\Omega)$ since by lemma \ref{lem:frac-unif-bound-trace-bis} implies
\begin{align*}
\|v_s-v\|_{W^{\eta,p}(\Omega)}
\leq C\|g_s-g_1\|_{W^{s,p}(\Omega)}
\xrightarrow{s\to1^-}0.
\end{align*}
In particular, we find that $\langle f_s, v_s -g_s\rangle_s
= \langle f_s, v -g_1\rangle_s\xrightarrow{s\to1^-} \langle f, v-g_1\rangle_1 $ by assumption on $(f_s)_s$
On the other hand, we find that
\begin{align*}
\big|\cE^{s,p}_\Omega(v_s,v_s)^{\frac{1}{p}}-\cE^{1,p}_\Omega(v,v)^{\frac{1}{p}}\big|
&\leq \cE^{s,p}_\Omega(v_s-v,v_s-v)^{\frac{1}{p}}+ \big|\cE^{s,p}_\Omega(v,v)^{\frac{1}{p}}-\cE^{1,p}_\Omega(v,v)^{\frac{1}{p}}\big|\\
&= \cE^{s,p}_\Omega(g_s-g_1,g_s-g_1)^{\frac{1}{p}}+ \big|\cE^{s,p}_\Omega(v,v)^{\frac{1}{p}}-\cE^{1,p}_\Omega(v,v)^{\frac{1}{p}}\big|\\
&\leq  \|g_s-g_1\|_{W^{s,p}(\Omega)}+ \big|\cE^{s,p}_\Omega(v,v)^{\frac{1}{p}}-\cE^{1,p}_\Omega(v,v)^{\frac{1}{p}}\big|.
\end{align*}
Given that $\Omega$ is Lipschitz, we have  $ \cE^{s,p}_\Omega(v,v)\xrightarrow{s\to1^-}\cE^{1,p}_\Omega(v,v)$, see for instance \cite{BBM01,Fog23}. From the fact that  $\|g_s-g_1\|_{W^{s,p}(\Omega)}\xrightarrow{s\to1^-}0$, it  follows that the recovery sequence $(v_s)_s$ satisfies
$\cE^{s,p}_\Omega(v_s,v_s)
\to  \cE^{1,p}_\Omega(v,v)$ as $s\to 1^-$. Therefore $v_s-g_s= v-g_1\in W^{1,p}_0(\Omega)\subset W^{s,p}_0(\Omega)$ and
\begin{align*}
\lim_{s\to1^-}\cJ^{s,p}_0(v_s)
&= \lim_{s\to1^-}\Big(\frac{1}{p}  \cE^{s,p}_\Omega(v_s, v_s)-\langle f_{s}, u_{s}-g_{s}\rangle_{s}\Big)\\
&=\frac{1}{p}\cE^{1,p}(v,v) -\langle f_{1}, u_{1}-g_{1}\rangle_1= \cJ^{1,p}_0(v).
\end{align*}
Let us now shows the liminf condition. Assume $v_s\to v$ in $W^{\eta,p}(\Omega)$ as $s\to 1^-$. The case $\liminf\limits_{s\to1^-} \cJ^{s,p}_0(v_s)= \infty$ is trivial. Assume $\liminf\limits_{s\to1^-}  \cJ^{s,p}_0(v_s)<\infty$. This implies that  $\liminf\limits_{s\to1^-}\cE^{s,p}_\Omega(v_s,v_s)<\infty.$
Up to a subsequence, we can assume that $\gamma^s_0(v_s)=\gamma^s_0(g_s)$ and
\begin{align*}
\sup_{s\in(0,1)}
\Big(\int_\Omega |v_s(x)|^p\d x+ \cE^{s,p}_\Omega(v_s,v_s)\Big)<\infty.
\end{align*}
In particular we have $v_s-g_s\to v_1-g_1$ and $v_s\to v_1$ in $L^p(\Omega)$ as $s\to 1^-$. According to Theorem \ref{thm:asymp-converg-interpo} we find that $v\in W^{1,p}(\Omega)$ and satisfies
\begin{align*}
&\|\gamma^1_0(v-g_1)\|_{L^p(\partial\Omega) }= \lim_{s\to1^-}\|\gamma^s_0(v_s-g_s)\|_{L^p(\partial\Omega) }=0,\\
&\liminf_{s\to1^-}\cE^{s,p}_\Omega(v_s,v_s)\geq
\liminf_{s_n\to1^-}\cE^{s_n,p}_\Omega(v_{s_n},v_{s_n})\geq \cE^{1,p}_\Omega(v,v).
\end{align*}
Hence, $\gamma^1_0(v)= \gamma^1_0(g_1)$, that is, $v-g_1\in W^{1,p}_0(\Omega)$.
Moreover, since $v_s-g_s \to v-g_1 $ in $L^p(\Omega)$ and $(f_s)_s$ asymptotically  weakly converges to $f_1$ as $s\to 1^-$   we have
$\langle f_s,v_s-g_s \rangle_s  \to  \langle f_1, v-g_1\rangle_1$. Thus  one gets the liminf condition
\begin{align*}
\liminf_{s\to1^-} \cJ^{s,p}_0(v_s)\geq \cJ^{1,p}_0(v).
\end{align*}
\end{proof}

\subsection{Proof of Theorem \ref{thm:opti-conv-dirch-frac-regio} and Theorem \ref{thm:optimal-conv-eigpair}}
\noindent We borrow the next result from \cite{Fog26}.  

\begin{theorem}[Linearizing the convergence]\label{thm:equiv-conv-in-form}
Let $(u_s)_s$ and $u$ be suitable functions such that the sequences
$(\cE^{s,p}_\Omega(u_s,u_s))_s$ and $(\cE^{s,p}_\Omega(u,u))_s$ are bounded, that is,
\begin{align*}
M := \sup_{s\in(0,1)} \big( \cE^{s,p}_\Omega(u_s,u_s) + \cE^{s,p}_\Omega(u,u) \big) < \infty .
\end{align*}
Then there exists a constant $C>0$, depending only on $M$ and $p$, such that
\begin{align*}
C^{-1}\cE^{s,p}_\Omega(u_s-u,u_s-u)
\leq
\big( \cE^{s,p}_\Omega(u_s, u_s-u) - \cE^{s,p}_\Omega(u, u_s-u) \big)
\leq
C\\cE^{s,p}_\Omega(u_s-u,u_s-u).
\end{align*}
\end{theorem}

The following result can  be found in \cite{Fog26}; see also  \cite[Lemma 2.8]{FeSa20}.
\begin{theorem}
\label{thm:asymp-conv-reg-frac-form}
Let $\Omega\subset \R^d$ be open,  $u,v\in W^{1,p}(\Omega)$ and the sequence $ (v_s)_s$ with  $v_s \in W^{s,p}(\Omega)$ such that  $v_s \to v$ in $L^p_{\loc}(\Omega)$ as $s\to1^-$ and
\begin{align*}
\sup_{s\in (1/p, 1)} \big(\|v_s\|^p_{L^p(\Omega)}+ \cE^{s,p}_\Omega(v_s,v_s)\big)<\infty.
\end{align*}
If $u\in W^{1,p}_0(\Omega)$ or $\Omega$ is an $W^{1,p}$-extension domain then there holds that
\begin{align*}
\lim_{s\to1^-}\cE^{s,p}_\Omega(u,  v_s)
= \cE^{1,p}_\Omega(u,v).
\end{align*}
\end{theorem}

\begin{proof}[\textbf{Proof of Theorem \ref{thm:opti-conv-dirch-frac-regio}}]
Let $\overline{g}_1 \in W^{1,p}(\Omega) \subset W^{s,p}(\Omega)$ such that $\gamma^1_0(\overline{g}_1) = g_1$. Note that $g_s - g_1 \in W^{s-\frac{1}{p},p}(\partial \Omega)$ since $\gamma^s_0\vert_{W^{1,p}(\Omega)} = \gamma^1_0$. Let $h_s \in W^{s,p}(\Omega)$ such that
\begin{align*}
\gamma^s_0 (h_s) = g_s - g_1 \text{ and}\qquad \|h_s\|_{W^{s,p}(\Omega)} \leq 1-s + \|g_s - g_1\|_{W^{s-\frac{1}{p},p}(\partial \Omega)}.
\end{align*}
Let us define $\overline{g}_s = h_s + \overline{g}_1 \in W^{s,p}(\Omega)$ so that $\gamma^s_0(\overline{g}_s) = g_s$ and
\begin{align*}
\|\overline{g}_s - \overline{g}_1\|_{W^{s,p}(\Omega)} \leq 1-s + \|g_s - g_1\|_{W^{s-\frac{1}{p},p}(\partial \Omega)} \xrightarrow{s\to1^-} 0.
\end{align*}
Hence, there is no loss of generality in assuming that  $g_s \in W^{s,p}(\Omega)$ such that
\begin{align*}
\|\gamma^s_0(g_s) - \gamma^1_0(g_1)\|_{L^p(\partial\Omega)} + \|g_s - g_1 \|_{W^{s,p}(\Omega)} \xrightarrow{s\to1^-} 0.
\end{align*}

Thus, 
 we can assume $\|g_s-g \|_{W^{s,p}(\Omega)}<1$ for $s\in (s_*, 1)$ so that Lemma \ref{lem:frac-unif-bound-trace-bis} implies
\begin{align*}
\sup_{s_*<s<1}\|g_s \|_{W^{s,p}(\Omega)}\leq 1+ \sup_{s_*<s<1}
\|g \|_{W^{s,p}(\Omega)}\leq 1+  C \|g\|^p_{W^{1,p}(\Omega)}<\infty.
\end{align*}
Since $(\|f_s \|_{(W^{s,p}_0(\Omega))'})_s$ is bounded by assumption,  we obtain the uniform estimate
\begin{align}\label{eq:uniform-bdd-f-g-D-regio}
M:= 	\sup_{s_*<s<1}\big(\|f_s \|_{(W^{s,p}_0(\Omega))'}+\|g_s \|_{W^{s,p}(\Omega)}\big)<\infty.
\end{align}
For $s\in (s_*,1)$,  the estimate \eqref{eq:sol-bd-regio-frac-D} implies  the uniform estimate
\begin{align}
\label{eq:unifo-bddness-regio-frac}
\begin{split}
\|u_s\|_{W^{s,p}(\Omega)}
&\leq C \big(\|f\|^{p'}_{(W^{s,p}_0(\Omega))'}+\|\gamma^s_0(g)\|^p_{ W^{s-\frac{1}{p},p}(\partial \Omega)}\big)^{1/p}\\
&\leq C\big(\|f_s \|^{\frac{1}{p-1}}_{(W^{s,p}_0(\Omega))'}+\|g_s \|_{W^{s,p}(\Omega)}\big)
\leq C
\end{split}
\end{align}
for a generic  constant  $C>0$ independent of $s$.
Accordingly, by the asymptotic compactness
Theorem \ref{thm:asymp-compact-frac}, there is $u\in W^{1,p}(\Omega)$ and subsequence $s_n\to 1^-$ such that
\begin{align*}
&\|u_{s_n} -u\|_{W^{\eta,p}(\Omega)}\xrightarrow{s_n\to 1^-}0
\qquad\text{for all $0\leq \eta<1$}, \\
&\|\gamma^{s_n}_0(u_{s_n}) -\gamma^1_0(u)\|_{W^{\tau-\frac{1}{p},p}(\partial\Omega)}\xrightarrow{s_n\to 1^-}0\qquad\text{for all $\frac1{p} <\tau<1$},\\
&\cE^{1,p}_\Omega(u,u)\leq \liminf_{n\to\infty}\cE^{s_n,p}_\Omega(u_{s_n}, u_{s_n}).
\end{align*}
Note that $\gamma^{s_n}_0 (u_{s_n})= \gamma^{s_n}_0(g_{s_n})$ and by Remark \ref{rem:trace-mono}, for $\frac{1}{p}<\tau<s_n<1$,  we have $\gamma^{\tau}_0 |_{W^{s_n,p}(\Omega)}= \gamma^{s_n}_0$  and $\gamma^{\tau}_0 |_{W^{1,p}(\Omega)}= \gamma^{1}_0$ so that $\gamma^{\tau}_0 (u_{s_n})= \gamma^{\tau}_0(g_{s_n})$. We deduce that
\begin{align*}
\|\gamma^1_0(u)-\gamma^1_0(g_1)\|_{W^{\tau-\frac{1}{p},p}(\partial\Omega)}
&\leq \|\gamma^{\tau}_0(u_{s_n}-u)\|_{W^{\tau-\frac{1}{p},p}(\partial\Omega)}+ \|\gamma^{\tau}_0(g_{s_n}-g_1)\|_{W^{\tau-\frac{1}{p},p}(\partial\Omega)}\\
&\leq \|u_{s_n} -u\|_{W^{\tau,p}(\Omega)}+ \|g_{s_n} -g_1\|_{W^{\tau,p}(\Omega)}\\
&\leq \|u_{s_n}-u\|_{W^{\tau,p}(\Omega)}+ C\|g_{s_n} -g_1\|_{W^{s_n,p}(\Omega)}
\xrightarrow{s_n\to 1^-}0,
\end{align*}
where the last line uses Lemma \ref{lem:frac-unif-bound-trace-bis}. Therefore, we deduce  that $\gamma^1_0(u)= \gamma^1_0(g_1)$, i.e., $u-g_1\in W^{1,p}_0(\Omega)$.
The strong convergence of $(u_{s_n})_n$ to $u$ in $L^p(\Omega)$ and the asymptotic weak convergence of $(f_{s_n} )_n $ yield $\langle f_{s_n}, u_{s_n}-g_{s_n} \rangle_{s_n }\to \langle f ,u-g_1\rangle_1$ and
\begin{align*}
\liminf_{n\to\infty}\cJ_0^{s_n,p}(u_{s_n})= \liminf_{n\to\infty}\Big(\frac{1}{p} \cE_\Omega^{s_n,p}(u_{s_n},u_{s_n} )-\langle f_{s_n}, u_{s_n}-g_{s_n}\rangle_{s_n} \Big) \geq  \cJ^{1,p}_0(u),
\end{align*}
where we recall that each $u_s$  uniquely  satisfies the minimization problem
\begin{align*}
\cJ_0^{s,p}(u_s)
&= \min_{v-g_s\in W^{s,p}(\Omega)} \cJ_0^{s,p}(v),
\\
\cJ_0^{s,p}(v)& = \frac{1}{p} \cE^{s,p}_\Omega(v,v)
-\langle f_s , v-g_s\rangle_s.
\end{align*}
Now, we want show that $u=u_1$.  Let $v\in g_1+W^{1,p}_0(\Omega)$.  According to Theorem \ref{thm:gamma-conv-regio-Jo} there is  a recovery sequence $v_s\in g_s+ W^{s,p}(\Omega)$ such that $v_s\to v$ in $L^p(\Omega)$ and
\begin{align*}
\lim_{s\to1^-}\cJ_0^{s,p}(v_{s})= \cJ^{1,p}_0(v).
\end{align*}
Since  each $u_{s_n}$ minimizes $\cJ^{s_n,p}_0$, i.e., $\cJ_0^{s_n,p}(u_{s_n})\leq
\cJ_0^{s_n,p}(v_{s_n}) $, we deduce that
\begin{align*}
\cJ^{1,p}_0(u)\leq \liminf_{n\to\infty}\cJ_0^{s_n,p}(u_{s_n})\leq
\liminf_{n\to\infty}\cJ_0^{s_n,p}(v_{s_n})= \cJ^{1,p}_0(v).
\end{align*}
We arrive at the conclusion that $\|u_{s_n}-u\|_{W^{\eta,p}(\Omega)}\to 0$ as $s_n\to1^-$ for $0\leq \eta<1$ and
\begin{align*}
\cJ^{1,p}_0(u)&= \min_{v-g_1\in W^{1,p}_0(\Omega)} \cJ^{1,p}_0(v).
\end{align*}
In other words, $u= u_1\in W^{1,p}(\Omega)$ is the unique weak solution to the Dirichlet $-\Delta_p u=f_1$ in $\Omega$ and $\gamma^1_0(u)=g_1$ on $\partial\Omega$. The uniqueness of $u_1$ implies  that $\|u_s-u_1\|_{W^{\eta,p}(\Omega)}\to 0$ as $s\to1^-$ for all $0\leq \eta<1$. In particular, $\|u_s-u_1\|_{L^p(\Omega)}\to 0$ as $s\to1^-$ so that  Theorem \ref{thm:asymp-conv-reg-frac-form} implies the asymptotic convergences
\begin{align*}
\cE^{s,p}_\Omega(u_1,u_s-u_1) \xrightarrow{s\to1^-}0
\quad\text{and}\quad
\cE^{s,p}_\Omega(u_1,u_s)
\xrightarrow{s\to1^-}
\cE^{1,p}_\Omega(u_1,u_1).
\end{align*}
Next,  since $\|g_s-g_1\|_{W^{s,p}(\Omega)}\xrightarrow{s\to1^-}0$, it
appears that $(u_s-g_s)_s$ converges to $u_1-g_1$ in $L^p(\Omega)$ and
$u_1-g_1\in W^{1,p}_0(\Omega)$. By the asymptotic weak convergence of
$(f_s)_s$ we obtain
\begin{align*}
\langle f_s ,u_s-g_s \rangle_s \to \langle f ,u_1-g_1 \rangle_1 \quad
\text{and} \quad \langle f_s ,u_1-g_1 \rangle_s \to \langle f ,u_1-g_1 \rangle_1.
\end{align*}
Since $(u_s)_s$ is asymptotically bounded see \eqref{eq:unifo-bddness-regio-frac},
we find that
\begin{align*}
|\cE^{s,p}_\Omega(u_s,g_s-g_1)| \leq \cE^{s,p}_\Omega(u_s,u_s)^{\frac{1}{p'}}
\cE^{s,p}_\Omega(g_s-g_1, g_s-g_1)^{\frac{1}{p}} \leq C\|g_s-g_1\|_{W^{s,p}(\Omega)}
\to0,
\end{align*}
as $s\to1^-$. Thence, it follows by linearity that
\begin{align*}
\cE^{s,p}_\Omega(u_s,u_s-u_1) &= \cE^{s,p}_\Omega(u_s,g_s-g_1)
+ \langle f_s ,u_s-g_s \rangle_s + \langle f_s ,g_1-u_1 \rangle_s
\xrightarrow{s\to1^-}0.
\end{align*}
As $u_s \in W^{s,p}(\Omega)$ is a weak solution and $u_1 \in W^{1,p}(\Omega)
\subset W^{s,p}(\Omega)$, using the test function $u_s-u_1 \in W^{s,p}(\Omega)$,
we have obtained altogether the convergences
\begin{align*}
\mathcal{E}^{s,p}_\Omega(u_1, u_s-u_1) \xrightarrow{s\to1^-}0 \quad
\text{and} \quad \mathcal{E}^{s,p}_\Omega(u_s, u_s-u_1) \xrightarrow{s\to1^-}0.
\end{align*}
Accordingly, Theorem \ref{thm:equiv-conv-in-form} implies
$\mathcal{E}^{s,p}_\Omega(u_s-u_1, u_s-u_1) \xrightarrow{s\to1^-}0$ and hence
\begin{align*}
\|u_s-u_1\|_{W^{s,p}(\Omega)} \xrightarrow{s\to1^-}0.
\end{align*}
\end{proof}
\begin{remark}

Let assume that  $f_s,f_1\in L^{p'}(\Omega)$ and $g_s,g_1\in W^{s,p}(\Omega)$ and  put $w_s= u_s-g_s$. Note that $w_s,w _1\in W^{s,p}_0(\Omega)$. The robust Poincar\'e inequality implies
\begin{align*}
\| u_s-u_1\|_{L^p(\Omega)}
&\leq \| u_s-g_s-(u_1-g_1)\|_{L^p(\Omega)}+\| g_s-g_1\|_{L^p(\Omega)}\\
&\leq C\cE^{s,p}_\Omega(u_s-g_s-(u_1-g_1),u_s-g_s-(u_1-g_1))^{1/p}+ \| g_s-g_1\|_{L^p(\Omega)}\\
&\leq C\cE^{s,p}_\Omega(u_s-u_1,u_s-u_1)^{1/p}
+ C\| g_s-g_1\|_{W^{s,p}(\Omega)},
\end{align*}
where $C$ is a generic constant independent of $s$.
It  follows that
\begin{align*}
\| u_s-u_1\|_{W^{s,p}(\Omega)}
&\leq C\cE^{s,p}_\Omega(u_s-u_1,u_s-u_1)^{1/p}+ C
\| g_s-g_1\|_{W^{s,p}(\Omega)}\\
&\leq C\Big(\cE^{s,p}_\Omega(u_s,w_s-w_1)-\cE^{s,p}_\Omega(u_1,w_s-w_1)\Big)^{1/p}+
C\| g_s-g_1\|_{W^{s,p}(\Omega)},
\end{align*}
where we used the fact that $\|u_s\|_{W^{s,p}(\Omega)}+ \|u_1\|_{W^{s,p}(\Omega)}\leq C$ and
\begin{align*}
\cE^{s,p}_\Omega(u_s-u_1,u_s-u_1)\leq C
\Big(\cE^{s,p}_\Omega(u_s-u_1,u_s-u_1)\Big).
\end{align*}
The weak formulation $\cE^{s,p}_\Omega(u_s, v)= \int_\Omega f_s v$, $v\in W^{s,p}_0(\Omega)$ implies
\begin{align*}
\cE^{s,p}_\Omega(u_s&,w_s-w_1)-\cE^{s,p}_\Omega(u_1,w_s-w_1)\\
&= \int_\Omega (f_s-f_1) (w_s-w_1)\d x+ \int_\Omega f_1 w_s\d x -\cE^{s,p}_\Omega(u_1,w_s)+ \cE^{s,p}_\Omega(u_1,w_1)-\cE^{1,p}_\Omega(u_1,w_1)\\
&\leq C\Big(\|f_s-f_1\|_{L^{p'}(\Omega)}+ \Big| \int_\Omega f_1 w_s\d x -\cE^{s,p}_\Omega(u_1,w_s)\Big| + \|(-\Delta)^s_{p,\Omega} u_1- (-\Delta)^1_{p,\Omega}u_1\|_{ (W^{1,p}_0(\Omega))'} \Big).
\end{align*}
Therefore we  find the estimate
\begin{align*}
\begin{split}
\| u_s-u_1\|_{W^{s,p}(\Omega)}
&\leq C\Big(\|f_s-f_1\|_{L^{p'}(\Omega)}+ \| g_s-g_1\|_{W^{s,p}(\Omega)} + \Big| \int_\Omega f_1 w_s\d x -\cE^{s,p}_\Omega(u_1,w_s)\Big|\\
&\quad +
\|(-\Delta)^s_{p,\Omega} u_1 - (-\Delta)^1_{p,\Omega}u_1\|_{ (W^{1,p}_0(\Omega))'}\Big).
\end{split}
\end{align*}
If we further assume that $w_s \in W^{1,p}_0(\Omega)$ so that  $\cE^{1,p}_\Omega(u_1, w_s) = \int_\Omega f_1 w_s$, we obtain
\begin{align*}
\| u_s-u_1\|_{W^{s,p}(\Omega)}
&\leq C\Big(\|f_s-f_1\|_{L^{p'}(\Omega)}+ \| g_s-g_1\|_{W^{s,p}(\Omega)}+ \|(-\Delta)^s_{p,\Omega} u_1- (-\Delta)^1_{p,\Omega}u_1\|_{ (W^{1,p}_0(\Omega))'}\hspace{-0.5ex}\Big).
\end{align*}

\end{remark}

\begin{proof}[\textbf{Proof of Theorem \ref{thm:optimal-conv-eigpair}}]
Recall that  each $(\lambda_s, \varphi_s)$ satisfies  $\|\varphi_s\|_{L^p(\Omega)}= 1$  and
\begin{align*}
\cE^{s,p}_\Omega(\varphi_s, v) = \lambda_s \int_\Omega |\varphi_s|^{p-2}\varphi_s v \d x \qquad \text{ for all }\,\,  v \in W^{s,p}_0(\Omega).
\end{align*}
This is equivalent to say that $\varphi_s$ minimizes of the fractional Rayleigh quotient,
\begin{align*}
\lambda_s =\cE^{s,p}_\Omega(\varphi_s, \varphi_s) = \min_{\substack{v \in W^{s,p}_0(\Omega) \\ v \neq 0}} \frac{\cE^{s,p}_\Omega(v,v)}{\|v\|_{L^p(\Omega)}^p} .
\end{align*}
Furthermore, the robust Poincar\'e-Friedrichs inequality \eqref{eq:rob-poincare-fried-regio} implies that
\begin{align*}
1=\|\varphi_s\|^p_{L^p(\Omega)}\leq B\cE^{s,p}_\Omega(\varphi_s, \varphi_s)= B \lambda_s,
\end{align*}
where we take into account $c^{-1}\leq \frac{C_{d,p,s}}{s(1-s)}\leq c$ for some $c= c(d,p)>1$ independent on $s$. For fixed $\varphi\in C^\infty_c(\Omega)\setminus\{0\}$, using Lemma \ref{lem:frac-unif-bound-trace-bis} we find that
\begin{align*}
\lambda_s= \frac{\cE^{s,p}_\Omega(\varphi, \varphi)}{\|\varphi\|^p_{L^p(\Omega)}}\leq C\frac{\|\varphi\|^p_{W^{1,p}(\Omega)}}{\|\varphi\|^p_{L^p(\Omega)}}=: C_\varphi.
\end{align*}
It appears that $\frac{1}{B}\leq \lambda_s \leq C_\varphi$ and  $\|f_s\|^{p'}= \|\varphi_s\|^p_{L^p(\Omega)}=1$  where we put $f_s= |\varphi_s|^{p-2} \varphi_s$.
 Thence, there exist a subsequence $s_n\to 1^-$, $f_1\in L^{p'}(\Omega)$ and  $\frac{1}{B}\leq \lambda_1\leq  C_\varphi$  such that  $f_{s_n}\rightharpoonup f_1$ in $L^{p'}(\Omega)$ and $\lambda_{s_n}\to \lambda_1$. It follows that $(\lambda_{s_n} f_{s_n})_n$ asymptotically  weakly converges to $\lambda_1f_1$. Accordingly Theorem \ref{thm:opti-conv-dirch-frac-regio} implies the optimal convergence
\begin{align*}
\lim_{n\to\infty} \|\varphi_{s_n} - \varphi_1\|_{W^{s_n,p}(\Omega)} = 0
\end{align*}
where $\varphi_1\in W^{1,p}_0(\Omega)$ satisfies
the weak formulation
\begin{align*}
\cE^{1,p}_\Omega (\varphi_1,v)= \int_\Omega f_1 v\d x\qquad\text{for all $v\in W^{1,p}_0(\Omega)$}.
\end{align*}
 Let show that $f_1= |\varphi_1|^{p-2}\varphi_1$.
Since $\|\varphi_{s_n}\|_{L^p(\Omega)} = 1$ for all $n$, the strong convergence $\varphi_{s_n}\to \varphi_1$ in $L^p(\Omega)$ directly implies that $\|\varphi_1\|_{L^p(\Omega)}=1$. Moreover, by Vitali's convergence theorem (or Lebesgue's Dominated Convergence Theorem via an $L^p$-dominating subsequence), it follows that $f_{s_n}= |\varphi_{s_n}|^{p-2}\varphi_{s_n}$ converges strongly to $|\varphi_1|^{p-2}\varphi_1$ in $L^{p'}(\Omega)$ as $n \to \infty$. Finally, since $f_{s_n}\rightharpoonup f_1 $ we deduce that $f_1= |\varphi_1|^{p-2}\varphi_1$.
\end{proof}


\noindent \textbf{Data Availability Statement (DAS)}:
Data sharing not applicable, no datasets were generated or analyzed during the current study.
\vspace{0.5mm}

\noindent {\small \textbf{Conflict of Interest:}
{The author declares that there is no conflict of interest regarding the publication of this paper.}

\end{document}